\documentclass[%
  jmp,
  amsmath,amssymb,
]{revtex4-1}

\usepackage[margin=0.75in]{geometry}

\usepackage{graphicx}
\usepackage{dcolumn}
\usepackage{bm}

\usepackage[english]{babel}
\usepackage[utf8]{inputenc}
\usepackage[T1]{fontenc}
\usepackage{mathptmx}
\usepackage{amsthm}

\usepackage{color}

\newtheorem{thm}{Theorem}[section]

\newtheorem{lem}[thm]{Lemma}
 \newtheorem{rem}[thm]{Remark}

\numberwithin{equation}{section}

 \DeclareMathOperator{\sig}{sig}

\usepackage{algorithmic}
\usepackage{algorithm}

\begin{document}

\title{Computing Floquet Hamiltonians with Symmetries}

\author{Terry Loring}
\email{loring@math.unm.edu}
 \affiliation{Department of Mathematics and Statistics\\
University of New Mexico\\
Albuquerque, New Mexico, 87123, USA}

\author{Fredy Vides}
 \email{fredy.vides@unah.edu.hn}
 
\affiliation{ 
Scientific Computing Innovation Center\\
School of Mathematics and Computer Science\\
Universidad Nacional Aut\'onoma de Honduras\\
Tegucigalpa, Honduras
}

\date{\today}

\begin{abstract}
Unitary matrices arise in many ways in physics, in particular as a
time evolution operator. For a periodically driven system one frequently
wishes to compute a Floquet Hamilonian that should  be a Hermitian operator
$H$ such that $e^{-iTH}=U(T)$ where $U(T)$ is the time evolution
operator at time corresponding the period of the system. That is,
we want $H$   to be equal to  $-i$ times a matrix logarithm of $U(T)$. If
the system has a symmetry, such as time reversal symmetry, one can
expect $H$ to have a symmetry beyond being Hermitian.

 We discuss
here practical numerical algorithms on computing matrix logarithms
that have certain symmetries which can be used to compute Floquet
Hamiltonians that have appropriate symmetries.  Along the way,
we prove some results on how a symmetry in the Floquet operator
$U(T)$ can lead to a symmetry in a basis of Floquet eigenstates.

\end{abstract}

\maketitle


\section{Introduction}

Unitary matrices arise in critical ways in many parts of physics,  for
example those representing time evolution. These are not always constructed
as an exponential $U=e^{iH}$, but can be result of an integration
of a differential equation or a product of exponentials
$U=e^{iH_{1}}e^{iH_{2}}\cdots e^{iH_{n}}$.

Given a physically meaningful unitary $U$, one generally wants either
to understand the eigenvectors and eignevalues, or to have a single
Hermitian $H$ so that $U=e^{iH}$. 
That is, we want a matrix logarithm $ iH$ of $U$.
Spectral theory tells us that
$U$ has a unitary diagonalization, since unitary implies normal. However, 
there is a lack of effective algorithms at present that find orthogonal eigenvectors
of normal matrices.  In the unitary case, we have at least two ways to proceed. 
If $U=e^{iH}$, then a unitary matrix that diagonalizes $H$ will diagonalize $H$.
Thus if we can do a good job computing the logarithm of a unitary
\cite{HighamEtAlSquareRootGroupPreserving} we have a method to unitarily
diagonalize the original unitary. The other method \cite{LoringLogOfUnitary}
is to compute a Schur decomposition of $U$ and simply zero-out any
stray terms about the diagonal in the triangular factor. Both methods
have limitations.

Tracking symmetries in systems is critical in present-day physics, for example
the symmetries corresponding to time-reversal.
The exact definition of time-reversal in a system with a time-varying
Hamiltonian is more complicated than for a system with a stationary
Hamiltonian, but in any case, all manner of symmetries can wind up
in the unitary that represents time-evolution. One important symmetry
for a unitary $U$ is that it be complex symmetric, $U^{\mathrm{T}}=U$,
which means there is a real symmetric $H$ so that $U=e^{iH}$. Here
the Schur decomposition does not work, at least with standard implementations.

Here we look at five symmetry classes of unitary matrices:  generic unitary, 
real orthogonal, complex symmetric, graded and self-dual.  The previous
method of using the Shur decomposition \cite{LoringLogOfUnitary} works
for three of these:  generic, self-dual and real orthogonal.  The self-dual
case uses a self-dual variation on the Shur decomposition.

The real orthogonal case was not implemented in \cite{LoringLogOfUnitary}.
We leave as an exercise for the reader utilizing the real form of the Schur decomposition
that reduces the problem to the two-by-two case.  One must be careful,
as a real logarithm is not always possible.  An exponential is in the connected
component of the identify, so if $K$ is real and skew symmetric then $U=e^{K}$ will
be real orthogonal with determinant one.   A real orthongonal matrix $U$ will have
a real logarithm if and only if  $\det(U)=1$.  An equivalent condition is that
$-1$ appear in the spectrum with even multiplicity.

The Schur decomposition is very easy to use to diagonalize a unitary matrix 
by a unitary--- one just zeros out the off-diagonal terms as needed to turn
upper-triangular into diagonal.  If one wants the logarithm, one simply
takes logarithm of the diagonal, a simple task.  The limitation is that, to our knowledge,
the Schur decomposition has only been developed for the general complex,
real orthogonal and self-dual unitary situations.   Thus we need to look 
for a different approach.

We look at a standard algorithm \cite[Ch.~11]{Higham}
for computing a matrix logarithm, which takes square root several times to
reduce to the situation where something like a power series can compute the logarithm. 
This algorithm requires the input matrix to have no spectrum on the
negative real axis, so at first glance this will restrict us to a unitary
matrix $U$ with $-1$ not in the spectrum.  In practice, this is not much of
a limitation.

In the algorithm, the eigenvalues at $-1$ are not stable, and so round-off
error is sufficient to make them quickly act as if they are not $-1$, unless
there is some manner of a $K$-theory index that is non-zero that
makes an eigenvalue at $-1$ stable.  This is fine, as the algorithm tends
to fail only when the logarithm does not exist anyway. 

The square-root algorithm we work with \cite{HighamEtAlSquareRootGroupPreserving} was designed to work for non-unitary matrices.  Indeed, the output can be far from
unitary when the input is very close to being unitary.  (The exact
condition $UU^\dagger = I$ almost never holds for a matrix
stored in floating point numbers.)  We make a simple modification
to the algorithm so that its output is very close to unitary when
the input is very close to unitary.

We find that this modified algorithm behaves quite nicely, even when applied
to most unitary matrices with $-1$ in the spectrum.  It will fail for a diagonal
unitary matrix with  $-1$ in the spectrum, but we hardly need a new
algorithm to  find a square root of a diagonal matrix.

Then we consider unitaries that are symmetric or have a grading symmetry.
The algorithm for square root can be further modified to enforce the needed symmetry
at every stage of the main iteration.  We are then able to create the desired 
matrix logarithm algorithm whose output has the corresponding symmetry.  
For example, for  complex symmetric unitary matrix, we expect, and can
compute, a logarithm of the form $iH$ with $H$ a real symmetric matrix.

For a recent physics paper that looks at unitaries with various symmetries,
and their spectrum, see \cite{Floquet-Block_with_symmetries}. For
a recent physics paper that looks at real symmetric $H$ with $e^{-iH}=U$
for $U$ a time evolution operator, see \cite{FloquetThermalization_with_Symmetries}.
A paper which computes an imaginary Hermitian $H$ so that $e^{-iH}=U$
for an orthogonal matrix $U$ is \cite{SurfaceHoppingAlgorithm}.

Our main concern is accuracy, not speed. An accurate log of a unitary
will be anti-Hermitian, have exponential close to the input matrix,
and be either real or purely imaginary when appropriate. An accurate
square root of a unitary $U$ will be a matrix $V$ with $V^{2}\approx U$
and $VV^\dagger\approx I$, and with $V$ having an additional symmetry
when appropriate. 

After applying the appropriate square root algorithm to $U$ several times, we can
apply a Pad\' e approximation, enforce the needed symmetries, and
end up with $H$ that is Hermitian and has additional symmetries.
We can apply a structured diagonalization algorithm to $H$ that 
then gives a  structured diagonalization of $U$.

Several interesting results corresponding to some prior work on the subject of structured matrix logarithms computation were reported in \cite{Cardoso2010ExponentialsOS} and \cite{Dieci1996ConsiderationsOC}, these papers address the need to get a unitary output when computing the square root of a unitary, although, to the best of our knowledge, the accuracy of practical applications of the methods proposed in these documents was not documented, and the cases of a small gaps at -1 were implicitly avoided.

The rest of the paper is organized as follows.
Section~\ref{sec:TimeEvoWithSym} looks at how symmetries in a time-varying
Hamiltonian lead to symmetries in the unitary propagator.  
Section~\ref{sec:Diagonalizing} takes a mathematical look at when structured logarithms and structured diagonalizations exist
for unitary matrices with symmetry.  
In Section~\ref{section:matrix_sqrt} some theoretical arguments and numerical methods for the computation of symmetry preserving matrix square roots, are presented. 
The numerical methods for the computation of symmetry preserving matrix logarithms are presented in Section~\ref{section:matrix_log}. In Section~\ref{sec:MatrixDiagAlg} some prototypical algorithms for symmetry preserving matrix diagonalization, that are based on the computational methods in Section~\ref{section:matrix_log}, are studied. 

\section{Time evolution, with symmetries}
\label{sec:TimeEvoWithSym}

\subsection{Floquet systems}

We consider very general Floquet systems, not simply the crystalline
systems, but also those based on quasicrystals and such. Thus we will
not have momentum space. Given a periodically varying Hamiltonian
$H(t)$, with time period $T$, the standard procedure is to consider
the Floquet operator, which is the time evolution operator $U(T)=U(0,T)$,
and then derive a Floquet Hamiltonian as some manner of a matrix logarithm.
We define this formally as
\begin{equation}
H_{F}=\frac{i}{T}\text{log}(U(T)).
\label{eq:H_F-by-log}
\end{equation}
We are assuming here a finite Hilbert space, so there will be at least a small
gap in the spectrum of $U$. If there are no symmetries to consider, then one can 
select a suitable branch of logarithm to use as $\text{log}$
in Equation~\ref{eq:H_F-by-log}.  For many symmetries, this will
not work unless one is lucky enough to have a gap at $-1$.
What one wants, in any case, is an operator $H_{F}$ such that
\begin{equation}
H_{F}^{\dagger}=H_{F}\label{eq:H_F-is-Hermitian}
\end{equation}
and
\begin{equation}
e^{-iTH_{F}}=U(T).\label{eq:H_F-is-a-log-of-U}
\end{equation}

Our focus is on disordered, defective, amorphous or quasicrystalline systems where 
numerical computations will be important, so we desire effective numerical
algorithms. The first step is computing the Floquet operator, by a time ordered
exponentiation, for example. We have little to add on the methods to use here. The
second step is computing the matrix logarithm of the Floquet operator.
This is our focus.

Symmetries play a key role in the study of Floquet topological insulators,
and arise in other aspects of physics. The definition of a symmetry
preserving Floquet system is generally defined in terms of symmetries
on the periodic path of Hamiltonians $H(t)$. This then manifests
itself as a symmetry on the Floquet operator. In theory then, the
Floquet Hamiltonian $H_{F}$ can be chosen to have an appropriate
symmetry. The challenge we take up here is finding algorithms to compute
$H_{F}$ from $U(T)$ so that $H_{F}$ has appropriate symmetries
and still satisfies Equations \ref{eq:H_F-is-Hermitian} and \ref{eq:H_F-is-a-log-of-U}.

To see how symmetries percolate from the $H(t)$ to the Floquet operator,
it is convenient to use the Suzuki-Trotter expansion 
\begin{equation}
U(T)=\lim_{N\rightarrow\infty}e^{-i(T/N)H(T)}\cdots e^{-i(T/N)H(2T/N)}e^{-i(T/N)H(T/N)}\label{eq:TotterExpansion}
\end{equation}
A nice overview of the theory of periodically driven systems, including
how to compute the Floquet operator, can be found in  \cite{NathanTopPDrivenSystems}. For large $N$ we have 
\[
e^{-i(T/N)H(T)}=e^{-i(T/N)H(0)}\approx I
\]
 and so we have also 
\begin{equation}
U(T)=\lim_{N\rightarrow\infty}e^{-i(T/N)H(T-T/N)}\cdots e^{-i(T/N)H(T/N)}e^{-i(T/N)H(0)}
\label{eq:TotterExpansion-alt}
\end{equation}
 which we will find useful when examining certain symmetries. In particular
notice this implies
\begin{equation}
U(T)^{\dagger}
=\lim_{N\rightarrow\infty}e^{i(T/N)H(0)}e^{i(T/N)H(T/N)}\cdots e^{i(T/N)H(T-T/N)}.
\label{eq:Totter_U_dagger}
\end{equation}

\subsection{Time Reversal Symmetry, $\mathcal{T}^{2}=I$ }

We consider now the case where $H(t)$ has time reversal symmetry with
$\mathcal{T}^{2}=I$. If we perform an appropriate orthogonal change
of basis we can assume the time reversal antilinear operator
$\mathcal{T}=\mathcal{T}^{-1}$
is just conjugation of vectors. For simplicity, we assume time-reversal
symmetry about $t=0$ so the assumption on the Hamiltonian is 
\[
H(-t)=\mathcal{T}\circ H(t)\circ\mathcal{T}.
\]
In terms of matrices, this becomes 
\[
H(-t)=H(t)^{*}
\]
or
\[
H(T-t)=H(t)^{*}
\]
 (The $*$ here means conjugation, so this is $H(-t)=\overline{H(t)}$
in mathematical notation.) We find then
\begin{align*}
U(T)^{*} 
& =\lim_{N\rightarrow\infty}
  e^{i(T/N)H(T)^{*}}\cdots e^{-i(T/N)H(2T/N)^{*}}e^{-i(T/N)H(T/N)^{*}}\\
 & =\lim_{N\rightarrow\infty}
   e^{i(T/N)H(0)}\cdots e^{i(T/N)H(T-2T/N)}e^{i(T/N)H(T-T/N)}
\end{align*}
and so, by Equation~\ref{eq:Totter_U_dagger}, we have shown $U(T)^{*}=U(T)^{\dagger}$.
In more familiar mathematical terms, this means $U(T)$ is a complex
symmetric unitary matrix, so 
\begin{equation}
U(T)=U(T)^{\top}
\label{eq:U(T)-symmetric}
\end{equation}
 where $\top$ denotes transpose.

\subsection{Time Reversal Symmetry, $\mathcal{T}^{2}=-I$ }

Should $H(t)$ have time reversal symmetry with $\mathcal{T}^{2}=-I$,
we can perform a similar analysis. Assume we have preformed the appropriate
orthogonal change of basis so that 
\begin{equation}
\mathcal{T}\left(\psi\right)=Z\overline{\psi}\label{eq:Time-reversal-operator-Fermionic}
\end{equation}
where
\[
Z=\left[\begin{array}{cc}
0 & I\\
-I & 0
\end{array}\right].
\]
Assuming time reversal about the origin, time reversal symmetry means
\[
H(-t)=\mathcal{T}\circ H(t)\circ\mathcal{T}
\]
but this now translates to
\[
H(-t)=ZH(t)^{*}Z.
\]
Let us use the dual operation
\[
X^{\sharp}=-ZX^{\top}Z.
\]
The dual operation is closely related to $\mathcal{T}$, as explained
in  \cite{LoringQMatrices}.  Time reversal symmetry now
becomes 
\[
H(-t)=\left(H(-t)\right)^{\dagger}=H(t)^{\sharp}.
\]
We find then
\begin{align*}
U(T)^{\sharp} & =\lim_{N\rightarrow\infty}e^{-i(T/N)H(T)^{\sharp}}\cdots e^{-i(T/N)H(2T/N)^{\sharp}}e^{-i(T/N)H(T/N)^{\sharp}}\\
 & =\lim_{N\rightarrow\infty}e^{-i(T-T/N)H(T/N)}e^{-i(T/N)H(T-2T/N)}\cdots e^{-i(T/N)H(0)}
\end{align*}
which means
\begin{equation}
U(T)^{\sharp}=U(T).\label{eq:U(T)-self-dual}
\end{equation}
Thus $U(T)$ is a unitary, self-dual matrix.

\subsection{Chiral Symmetry }

A simple form of Chiral symmetry we might find for some periodic systems
is
\[
\Gamma H(t)\Gamma=-H(-t)
\]
where $\Gamma$ is a unitary matrix that squares to one. A common
choice would be
\begin{equation}
\Gamma=\left[\begin{array}{cc}
-I_{N/2} & 0\\
0 & I_{N/2}
\end{array}\right]
\label{eq:Gamma_def}
\end{equation}
which of course means $N$ is even. 

Assuming this symmetry, we find
\[
\Gamma e^{X}\Gamma=\Gamma\left(\sum\frac{1}{n!}X^{n}\right)\Gamma=\sum\frac{1}{n!}\left(\Gamma X\Gamma\right)^{n}=e^{\Gamma X\Gamma}
\]
and 
\begin{align*}
\Gamma U(T)\Gamma & =\lim_{N\rightarrow\infty}e^{i(T/N)H(-T)}\cdots e^{i(T/N)H(-2T/N)}e^{i(T/N)H(-T/N)}\\
 & =\lim_{N\rightarrow\infty}\left(e^{-i(T/N)H(-T/N)}e^{-i(T/N)H(-2T/N)}\cdots e^{-i(T/N)H(-T)}\right)^{\dagger}\\
 & =\left(\lim_{N\rightarrow\infty}e^{-i(T/N)H(T-T/N)}e^{-i(T/N)H(T-2T/N)}\cdots e^{-i(T/N)H(0)}\right)^{\dagger}
\end{align*}
and finally
\begin{equation}
\Gamma U(T)\Gamma=U(T)^{\dagger}.
\label{eq:Chiral_simmetry_def}
\end{equation}

\section{Symmetry preserving Effective Hamiltonians}
\label{sec:Diagonalizing}

\subsection{Principal logarithms}

Since the Floquet operator $U(T)$ is unitary, in the finite dimensional
case it is a normal matrix and the spectral theorem applies. In terms
of matrices the spectral theorem for normal matrices tells us
\[
U(T)=QDQ^{\dagger}
\]
where $Q$ is another unitary matrix and $D$ is diagonal with diagonal
entries on the unit circle. In the equivalent picture of Floquet eigenstates
the statement is that there is a orthogonal basis of $\psi_{j}$ with
\[
U(T)\psi_{j}=e^{i\alpha_{j}}\psi_{j}, \quad (- \pi < \alpha_j \leq\pi ).
\]
An effective Hamiltonian is then
\begin{equation}
H_{F}=QAQ^{\dagger}\label{eq:H_F-by-spectral-decomp}
\end{equation}
where $A$ is diagonal with diagonal entries $\alpha_{1},\dots,\alpha_{N}$.

If we can find an additional symmetry in the orthogonal basis we should
be able to find an effective Hamiltonian with the desired symmetry.
This implication does not lead us to an algorithm, but it is illuminating.

It is the possibility of $-1$ in the spectrum that makes even the theory
more difficult.  If  $-1$ is not in the spectrum of $U(T)$ we can
use the functional calculus which, as has been noted many times before,
has a tendency to preserve symmetries.  Again, this does not lead us to an
algorithm.  It does provide a way to consider also non-unitary matrices, with 
symmetries, and how their principal logarithms pick up related
symmetries.

For a given matrix $X \in \mathbb{C}^{n\times n}$,  a logarithm 
$L\in \mathbb{C}^{n\times n}$ of $X$ is any matrix such that $e^L = X$.  We
assume that $X$ has no eigenvalues on $\mathbb{R}^{-} = (-\infty,0]$ so that the existence of a unique principal logarithm is assured as shown in the following theorem.

\begin{thm} \label{thm:Higham_matrix_log} \cite[Theorem 1.31.]{Higham}
Let $X \in \mathbb{C}^{n\times n}$ have no eigenvalues on $\mathbb{R}^{-}$. There is a unique logarithm $L$ of $X$  all of whose eigenvalues lie in the strip 
$\{z\in \mathbb{C} : -\pi < \mathrm{Im}(z) < \pi\}$.
We refer to $H$ as the principal logarithm of $X$ and write $H = \log (X)$. 
\end{thm}

It is easy to check (using power series) that for any $X$ with spectrum
avoiding $\mathbb{R}^{-}$ that if $L$ is the principal logarithm of $X$ then
$ e^{L^{\dagger}}=X^{\dagger} $
so uniqueness tells us
\begin{equation} \label{log_of_adjoint}
\log(X^{\dagger})=\log(X)^{\dagger}.
\end{equation}

A useful formula for theoretical work is to express the principal logarithm via the
analytic functional calculus.  This works for bounded operators just as well as
for matrices.  Still assuming  $X$ have no spectrum on $\mathbb{R}^{-}$, we have
\[
\log (X) = \oint_\gamma\log(z)\left(z-X\right)^{-1}\,dz
\]
where $\gamma$ is a positively oriented contour that encloses the spectrum of
$X$ and is within the complement of  $\mathbb{R}^{-}$.

The following holds true for bounded operators on Hilbert space, provided one
uses a sensible definition of transpose, the dual operation and the
unitary symmetry $\Gamma$.  For matrices, these were defined in 
\S\ref{sec:TimeEvoWithSym}.  In the matrix case, Theorem~\ref{thm:FunCalcSymmetry}
is substantially the same as \cite[Theorem 3.1]{HighamEtAlSquareRootGroupPreserving}.

\begin{thm}  \label{thm:FunCalcSymmetry}
Assume $X \in \mathbb{C}^{n\times n}$ have no eigenvalues on $\mathbb{R}^{-}$ and
that $L$ is the principal logarithm of $X$.  Then the following are all true.
\begin{enumerate}
\item If $X$ is unitary then $L^{\dagger}=-L$.
\item If $X$ is complex symmetric then $L$ is complex symmetric.
\item If $X$ is real then $L$ is real.
\item If $X$ is self-dual then $L$ is self-dual.
\item If $\Gamma X\Gamma=X^{\dagger}$ then $\Gamma L\Gamma=L^{\dagger}$.
\end{enumerate}
\end{thm}

\begin{proof}
(1) 
When $X$ is unitary, it is normal and so the holomorphic functional
calculus equals the continuous functional calculus. That is, we
can compute  $L=\log(X)$ as in Equation~\ref{eq:H_F-by-spectral-decomp}
where it is evident that $L^\dagger = -L$ (notice in defining an effective
Hamiltonian we would use $-iL$).

For the other parts of the proof we need to use the holomorphic functional
calculus, so select $\gamma$, a
positively oriented loop that encloses the spectrum of $U$ and avoids
 $\mathbb{R}^{-}$.  Then
\begin{equation} \label{CauchyFormula}
 L=\oint_{\gamma}\log(z)\left(z-X\right)^{-1}\,dz .
\end{equation}

(2)  The transpose operation pulls inside a path integral, so
\begin{equation} \label{TransposeLog}
\log(X)^{\top}=\oint_{\gamma}\log(z)\left(z-X^{\top}\right)^{-1}\,dz=\log(X^{\top}).
\end{equation}
Thus the assumption $X^\top = X$ leads to $\log(X)^{\top} = \log(X)$.

(3) Now assume $X$ is real, so that $X^\dagger = X^\top$.  Then by Equation~\ref{TransposeLog} we have $\log(X^\dagger) = \log(X)^{\top}$, while by
Equation~\ref{log_of_adjoint} we have $\log(X^\dagger) = \log(X)^\dagger$.  Thus
$ \log(X)^{\top}= \log(X)^\dagger$ and so is also real.

(4)  This proof is similar to the proof of (2), relying on the formula
\begin{equation} \label{DualLog}
\log(X)^{\sharp}=\oint_{\gamma}\log(z)\left(z-X^{\sharp}\right)^{-1}\,dz=\log(X^{\sharp}).
\end{equation}

(5)  
Now consider the additional assumption that $\Gamma X\Gamma=X^{\dagger}$. Then
\begin{align*}
\Gamma L\Gamma 
 & =\oint_{\gamma}\Gamma\log(z)\left(z-X\right)^{-1}\Gamma\,dz\\
 & =\oint_{\gamma}\log(z)\left(z-\Gamma X\Gamma\right)^{-1}\,dz\\
 & =\oint_{\gamma}\log(z)\left(z-X^{\dagger}\right)^{-1}\,dz\\
 & =\log(X^{\dagger}) .
\end{align*}

 Thus we have proven $ \Gamma L\Gamma=L^{\dagger}$.

\end{proof}

In the following subsections, we examine the case of a unitary with $-1$
in the spectrum, assuming an additional symmetry.   We limit the
discussion to both cases of time reversal symmetry and Chiral
symmetry.

\subsection{Time Reversal Symmetry, $\mathcal{T}^{2}=I$ }

If the Floquet operator $U(T)$ is complex symmetric, as well as unitary,
we can use another variation on the spectral theorem. If we let
\[
X=\frac{1}{2}\left(U(T)^{\dagger}+U(T)\right)
\]
 and
\[
Y=\frac{i}{2}\left(U(T)^{\dagger}-U(T)\right)
\]
then one can easily verify that $X$ and $Y$ are commuting real symmetric
matrices with $U(T)=X+iY$. The spectral theorem for commuting real
symmetric matrices tells us 
\[
X=QGQ^{\top},\quad Y=QMQ^{\top}
\]
for some real orthogonal matrix $Q$ and diagonal matrices $G$ with
real diagonal elements $\gamma_{1},\dots,\gamma_{N}$ and $M$ with real diagonal elements
$\mu_{1},\dots,\mu_{N}$. If we set $D=G+iM$ then we will have achieved
\[
U(T)=QDQ^{\dagger}
\]
with $D$ diagonal with diagonal entries on the unit circle and now
$Q$ a real orthogonal matrix. This means we can select the Floquet
eigenstates to be real. Also Equation~\ref{eq:H_F-by-spectral-decomp}
now tells us that the effective Hamiltonian can be taken to be a real
symmetric matrix.

\subsection{Time Reversal Symmetry, $\mathcal{T}^{2}=-I$ }

If the Floquet operator $U(T)$ is self dual, as well as unitary,
we can write $U(T)=A+iB$ with $A$ and $B$ commuting self-dual,
Hermitian matrices. There is a spectral theorem for commuting Hermitian
self-dual matrices \cite[Theorem~2.4]{LoringQMatrices} that tells us 
\[
X=QGQ^{\dagger},\quad Y=QMQ^{\dagger}
\]
for $Q$ a unitary with $Q\circ\mathcal{T}=\mathcal{T}\circ Q$ and
where 
\[
G=\left[\begin{array}{cc}
G_{0} & 0\\
0 & G_{0}
\end{array}\right],\quad M=\left[\begin{array}{cc}
M_{0} & 0\\
0 & M_{0}
\end{array}\right]
\]
and $G_{0}$ and $M_{0}$ are diagonal with real diagonals $\gamma_{1},\dots,\gamma_{N/2}$
and $\mu_{1},\dots,\mu_{N/2}$. Here the $\mathcal{T}$ we have in
mind is as defined in Equation~\ref{eq:Time-reversal-operator-Fermionic}.
Again we set $D=G+iM$ and we have 
\[
U(T)=QDQ^{\dagger}
\]
for $Q$ unitary with $Q\circ\mathcal{T}=\mathcal{T}\circ Q$ and
\[
D=\left[\begin{array}{cc}
D_{0} & 0\\
0 & \overline{D_{0}}
\end{array}\right]
\]
with $D_{0}$ diagonal with diagonal $e^{i\alpha_{1}}.\dots,e^{i\alpha_{N/2}}$
and the $\alpha_{j}$ real. The condition $Q\circ\mathcal{T}=\mathcal{T}\circ Q$
means that column $j+N/2$ of $Q$ is the result of $\mathcal{T}$applied
to column $j$ of $Q$. Thus we can find an orthogonal basis of Floquet
eigenstates that appear in time reversal pairs,
$ \psi_{j},\mathcal{T}(\psi_{j}) $,
with 
\[
U(T)\psi_{j}=e^{i\alpha_{j}}\psi_{j}
\]
and
\[
U(T)\mathcal{T}(\psi_{j})=e^{-i\alpha_{j}}\mathcal{T}(\psi_{j}).
\]
This also means we can find an effective Hamiltonian that is a self-dual,
Hermitian matrix.

\subsection{Chiral Symmetry}

Assume $\Gamma$ is a unitary matrix with $\Gamma^{2}=I$. We only
concern ourselves with the case where the $+1$ and $-1$ eigenspaces
for $\Gamma$ are of the same dimension. Up to a unitary change of
basis, we are able to assume then 
\[
\Gamma=\left[\begin{array}{cc}
I_{N/2} & 0\\
0 & -I_{N/2}
\end{array}\right]
\]
where $N$ is even. 

The following must certainly have appeared somewhere in the literature.
As the proof is short, we include it.
\begin{thm}
\label{thm:Chiral_spectral_decomposition}
Suppose $U$ is a unitary matrix with 
\[
\Gamma U\Gamma=U^{\dagger}.
\]
Then there are matrices $Q$ and $D$ such that 
\[
D=\left[\begin{array}{cc}
D_{11} & D_{12}\\
D_{21} & D_{22}
\end{array}\right]
\]
has $N/2$-by-$N/2$ blocks that are diagonal, $Q$ is unitary with
\[
\Gamma Q\Gamma=Q
\]
and 
\[
U=QDQ^{\dagger}.
\]
Moreover, we can do all the above with the added condition that
the $D_{jk}$ are real matrices.
\end{thm}

\begin{proof}
We will prove the first statement by induction on $N$. In the case case, $N=2$
and we can use $Q=I$ and $D=U$.

Now assume that $U$ is a unitary with this symmetry, that $N\geq4$,
and that the theorem is true for unitaries of size $(N-2)$-by-$(N-2)$. 
The spectral theorem for normal matrices applies to $U$. Thus, if
\[
U\boldsymbol{v}=\lambda\boldsymbol{v}
\]
for unit vector $\boldsymbol{v}$ then also $U^{\dagger}\boldsymbol{v}=\overline{\lambda}\boldsymbol{v}.$
Thus $\Gamma U\Gamma\boldsymbol{v}=\overline{\lambda}\boldsymbol{v}$
which implies 
\[
U\Gamma\boldsymbol{v}=\overline{\lambda}\Gamma\boldsymbol{v}.
\]
Of course $\lambda$ must be on the unit circle.

First suppose there is some unit eigenvector $\boldsymbol{v}$ for
$U$ for which $\lambda$ is not real. Then $\Gamma\boldsymbol{v}$
is an eigenvector for  $\overline{\lambda}$, which is not equal to
$\lambda$, which implies that $\Gamma\boldsymbol{v}$ is orthogonal
to $\boldsymbol{v}$.
Consider the two vectors
\[
\boldsymbol{w}_{\pm}=\frac{1}{\sqrt{2}}\boldsymbol{v}\pm\frac{1}{\sqrt{2}}\Gamma\boldsymbol{v}.
\]
These are both unit vectors, and
\[
\Gamma\boldsymbol{w}_{\pm}=\pm\boldsymbol{w}_{\pm}
\]
so $\boldsymbol{w}_{+}$ is in $\mathbb{C}^{N/2}\oplus0$ while $\boldsymbol{w}_{-}$
is in $0\oplus\mathbb{C}^{N/2}$. Thus we can find an orthogonal basis
of the form
\[
\boldsymbol{b}_{1},\dots,\boldsymbol{b}_{N/2},\boldsymbol{c}_{1},\dots,\boldsymbol{c}_{N/2}
\]
with $\boldsymbol{b}_{1}=\boldsymbol{w}_{+}$, $\boldsymbol{c}_{1}=\boldsymbol{w}_{-}$,
each $\boldsymbol{b}_{k}$ in $\mathbb{C}^{N/2}\oplus0$ and each
$\boldsymbol{c}_{k}$ in $0\oplus\mathbb{C}^{N/2}$. Notice that $\{\boldsymbol{b}_{1},\boldsymbol{c}_{1}\}$
spans an invariant subspace for $U$, and so also the span of 
\[
\boldsymbol{b}_{2},\dots,\boldsymbol{b}_{N/2},\boldsymbol{c}_{2},\dots,\boldsymbol{c}_{N/2}
\]
is an invariant subspace for $U$. It is also an invariant subspace
for $\Gamma$, and the restriction of $\Gamma$ to that subspace has 
still balanced $+1$ and $-1$ eigenspaces. By the induction
hypothesis, we can find a new basis 
\[
\tilde{\boldsymbol{b}}_{2},\dots,\tilde{\boldsymbol{b}}_{N/2},\tilde{\boldsymbol{c}}_{2},\dots,\tilde{\boldsymbol{c}}_{N/2}
\]
for this subspace so that each $\text{span}(\tilde{\boldsymbol{b}}_{k},\tilde{\boldsymbol{c}}_{k})$
is an invariant subspace for $U$. Thus
\[
\boldsymbol{b}_{1},\tilde{\boldsymbol{b}}_{2},\dots,\tilde{\boldsymbol{b}}_{N/2},\boldsymbol{c}_{1},\tilde{\boldsymbol{c}}_{2},\dots,\tilde{\boldsymbol{c}}_{N/2}
\]
gives the desired basis needed to define $Q$.

The remaining case is where $U$ has eigenvalues only at $\pm1$.
This means $U^{\dagger}=U$ and so $\Gamma$ commutes with $U$. Commuting
Hermitian matrices can be diagonalized by one unitary $Q$. Commuting
with $\Gamma$ means $\Gamma Q\Gamma=Q$. Since $D=Q^{\dagger}UQ$ will
be diagonal, it will certainly have the desired diagonal blocks.

To get the second statement, it suffices to understand the 2-by-2 case.  
Given
\[
U=\left[\begin{array}{cc}
a & b\\
c & d
\end{array}\right]
\]
with $\Gamma U\Gamma=U^{\dagger},$ since
\[
\Gamma U\Gamma=\left[\begin{array}{cc}
a & -b\\
-c & d
\end{array}\right]
\]
we see that $a$ and $d$ are already real, and $c=\overline{b}$.
If $|\gamma|=1$ is the phase such that $\gamma\overline{b}$ is real,
then the equation
\[
\left[\begin{array}{cc}
1 & 0\\
0 & \gamma
\end{array}\right]\left[\begin{array}{cc}
a & b\\
\overline{b} & d
\end{array}\right]\left[\begin{array}{cc}
1 & 0\\
0 & \gamma
\end{array}\right]^{\dagger}=\left[\begin{array}{cc}
a & \overline{\gamma}b\\
\gamma\overline{b} & d
\end{array}\right]
\]
show how to alter $Q$ to still commute with $\Gamma$ but not make
the blocks real as well as diagonal. 
\end{proof}

If there are no eigenvalues of $U$ equal to $\pm1$ then the proof
just given indicates how to select eigenvectors in pairs of the form
$\boldsymbol{v},\Gamma\boldsymbol{v}$.
The eigenspaces of $U$ for $\lambda=\pm 1$ need special care, and it
is simply not always possible to form the pairs desired.  We need to assume
more about $U$. 

Let us now consider the $\lambda = 1$ eigenspace $E_{+1}$ of $U$.
If $U\boldsymbol{v}=\boldsymbol{v}$ then $U\Gamma\boldsymbol{v}=\Gamma\boldsymbol{v}$
and so $E_{+1}$ is an invariant subspace of $\Gamma$.
Thus we can decompose $E_{+1}$ into
\[
E_{+1}=E_{+1}^{+}\oplus E_{+1}^{-}
\]
where $\Gamma$ acts like $1$ on vectors in $E_{+1}^{+}$ and like
$-1$ on vectors in $E_{+1}^{-}$. If 
\[
\dim(E_{+1}^{+})=\dim(E_{+1}^{-})
\]
then we can get structured eigenvectors within $E_{+1}$ as follows. Take
any orthonomal basis
$ \boldsymbol{e}_{1},\dots,\boldsymbol{e}_{k} $
of $E_{+1}^{+}$ and any orthonomal basis
$\boldsymbol{f}_{1},\dots,\boldsymbol{f}_{k}$ of $E_{+1}^{-}$. Set
\[
\boldsymbol{v}_{j}=\frac{1}{\sqrt{2}}\boldsymbol{e}_{j}+\frac{1}{\sqrt{2}}\boldsymbol{f}_{j}
\]
and notice
\[
\Gamma\boldsymbol{v}_{j}=\frac{1}{\sqrt{2}}\boldsymbol{e}_{j}-\frac{1}{\sqrt{2}}\boldsymbol{f}_{j}.
\]
Then
$
\boldsymbol{v}_{1},\Gamma\boldsymbol{v}_{1},\dots,\boldsymbol{v}_{k},\Gamma\boldsymbol{v}_{k}
$
is a basis of the eigenspace of $U$ for $+1$.
 
The eigenspace for  $\lambda = -1$ is dealt with similarly.
The eigenvalues that are not real lead to structured pairs as discussed in the proof of Theorem~\ref{thm:Chiral_spectral_decomposition}.

In practice, computing the spaces $E_{-1}^{\pm}$ is not a numerically
stable operation. Forturnately, there is an invariant (coming from
$K$-theory, ultimately) that is stable to compute and that tells
us if $\dim(E_{-1}^{+})$ equals $\dim(E_{-1}^{-})$ .
We stick with the special case where $\Gamma$ has eigenvalues
$\pm 1$ of equal multiplicity so that the obstruction for
the $E_{-1}$ vanishes exactly when the the obstruction for
the $E_{+1}$ vanishes, and so we need only a single invariant.

This invariant uses the signature. For an invertible Hermitian matrix
$X$, the signature of $X$, denoted $\sig(X)$, is the number of
positive eigenvalues of $X$ minus the number of negative eigenvalues
of $X$, each counted with multiplicity.

\begin{lem}
\label{lem:Sig_zero_balnance_Espaces}If $U$ is a unitary with $\Gamma U\Gamma=U^{\dagger}$
then $U\Gamma$ is Hermitian and invertible,
\begin{equation} \label{eqn:AIII_index}
\frac{1}{2}\sig(U\Gamma)
\end{equation}
 is an integer, and 
\[
\frac{1}{2}\sig(U\Gamma)=0
\]
if and only if 
\begin{equation}
\dim(E_{-1}^{+})=\dim(E_{-1}^{-})\label{eq:balanced_E_-1}
\end{equation}
for the subspaces $E_{-1}^{\pm}$ defined as above.
\end{lem}

\begin{proof}
We know $\Gamma U^{\dagger}=U\Gamma$ and so
\[
\left(U\Gamma\right)^{\dagger}=\Gamma U^{\dagger}=U\Gamma.
\]
Notice $U\Gamma$ is a product of unitary matrices, so is unitary,
so is very much invertible. Since $N$ is even, the signature of $U\Gamma$
must be even. We can decompose $U$ as $U=QDQ^{\dagger}$ for $D$
with diagonal blocks and $Q$ a unitary with $\Gamma Q\Gamma=Q$. Then
\[
Q^{\dagger}U\Gamma Q=Q^{\dagger}UQ\Gamma=D\Gamma.
\]
 The signature is invariant under unitary conjugation, so
\[
\frac{1}{2}\sig(U\Gamma)=\frac{1}{2}\sig(D\Gamma).
\]
We can thus reduce to the case of $D$ with diagonal blocks. As both the signature
and Equation~\ref{eq:balanced_E_-1} respect direct sums, we can
reduce further to the two-by-two case. 

If $N=2$ then 
\[
D=\left[\begin{array}{cc}
\cos(\theta) & -\sin(\theta)\\
\sin(\theta) & \cos(\theta)
\end{array}\right]
\]
for some $\theta$ in $[0,2\pi]$ or 
\[
D=\left[\begin{array}{cc}
1 & 0\\
0 & 1
\end{array}\right]
\]
or 
\[
D=\left[\begin{array}{cc}
-1 & 0\\
0 & -1
\end{array}\right].
\]
The theorem is easy to check in all three cases.
\end{proof}

\begin{lem}
Suppose $U$ is a unitary with $\Gamma U\Gamma=U^{\dagger}$.  If there
is a Hermitian matrix $H$ with 
$
\Gamma H=-H\Gamma
$
and
$
e^{-iH}=U
$
then
\[
\frac{1}{2}\sig(U\Gamma)=0 .
\]
\end{lem}

\begin{proof}
We note first
\[
\frac{1}{2}\sig(I\Gamma)=0
\]
 and that 
\[
\frac{1}{2}\sig(e^{-itH}\Gamma)
\]
varies continuously in $t$. As this is always an integer it is constant,
so
\[
0=\frac{1}{2}\sig(e^{-itH}\Gamma)=\frac{1}{2}\sig(U\Gamma).
\]
\end{proof}

\begin{rem}
If a unitary $U$ with $\Gamma U \Gamma = U^\dagger$ has a square root
$V$ that is also a unitary with $\Gamma V \Gamma = V^\dagger$,
then $V$ has a structured logarithm and so also does $U$.  Thus
$\tfrac{1}{2}\sig(U\Gamma) \neq 0$ is also an obstruction to
finding structured square roots.
\end{rem}

Now we see this index gives us the only obstruction to finding
a structured logarithm.

\begin{thm}
\label{thm:ChiralLogarithm}
If $U$ is a unitary with $\Gamma U\Gamma=U^{\dagger}$ and
\[
\frac{1}{2}\sig(U\Gamma)=0
\]
then there is a Hermitian matrix $H$ with 
$\Gamma H=-H\Gamma$ and $e^{-iH}=U$.
\end{thm}

\begin{proof}
In the special case where $-1$ is not in the spectrum,we just use the standard branch of logarithm and define $H=i\log(U)$.  Theorem~\ref{thm:FunCalcSymmetry} tells us
that $H^\dagger = H$ and $\Gamma H \Gamma = -H$. 
If $-1$ is in the spectrum of $U$, one cannot use any other branch of
logarithm as these fail to have the correct symmetry and $H$ will as well.
We need to carefully split apart the eigenspace of $-1$ and assign some parts the new eigenvalue $\pi$, and assign other parts the eigenvalue $-\pi$ .

Once more we decompose $U$ as $U=QDQ^{\dagger}$ for $D$ real and block diagonal,
and $Q$ a unitary with $\Gamma Q\Gamma=Q$. We are assuming $\frac{1}{2}\mathrm{\text{sig}}(U\Gamma)=0$,
or equivalently $\frac{1}{2}\mathrm{\text{sig}}(D\Gamma)=0$. Next
we find $K$ with  $\Gamma K=-K\Gamma$ and $ e^{-iK}=D$.
 This can be done in (graded) two-by-two blocks, which we can assume
only come in the form 
\[
B_{\theta}=\left[\begin{array}{cc}
\cos(\theta) & -\sin(\theta)\\
\sin(\theta) & \cos(\theta)
\end{array}\right]
\]
for various $\theta$ in $(-\pi,\pi]$. Since the index invariant
is zero, by Lemma~\ref{lem:Sig_zero_balnance_Espaces}, we can combine
the one-by-one blocks into $B_{\pi}$ and $B_{-\pi}$. Since
\[
B_{\theta}=Q_{\theta}D_{\theta}Q_{\theta}^{\dagger}
\]
for
\[
D_{\theta}=\left[\begin{array}{cc}
\cos(\theta)-i\sin(\theta) & 0\\
0 & \cos(\theta)+i\sin(\theta)
\end{array}\right]
\]
and
\[
Q_{\theta}=\frac{1}{\sqrt{2}}\left[\begin{array}{cc}
-i & i\\
1 & 1
\end{array}\right]
\]
 we can define
\[
H_{\theta}=Q_{\theta}\left[\begin{array}{cc}
-\theta & 0\\
0 & \theta
\end{array}\right]Q_{\theta}^{\dagger}=\left[\begin{array}{cc}
0 & i\theta\\
-i\theta & 0
\end{array}\right].
\]
Taking blocked direct sums of these and we find a block-diagonal real
matrix $K$ of the form
\[
K=\left[\begin{array}{cc}
0 & K_{0}\\
-K_{0} & 0
\end{array}\right]
\]
 with $e^{-i K}=D$. The desired $H$ is $QKQ^{\dagger}$.
\end{proof}

\begin{lem}
\label{lem:Chiral_implies_signature_zero} If the Floquet Hamiltonian
$H(t)$ has chiral symmety $\Gamma H(t)\Gamma=-H(-t)$, for $\Gamma$
with eigenspaces for $-1$ and for $+1$ of the same dimension, then
\[
\frac{1}{2}\sig(U(T)\Gamma)=0.
\]
\end{lem}

\begin{proof}
For $t$ in $[0,T/2]$ we have
\[
U(t,T-t)=\lim_{N\rightarrow\infty}e^{-i\Delta H(t_{N})}\cdots e^{-i\Delta H(t_{2})}e^{-i\Delta H(t_{1})}
\]
where $\Delta=\frac{T-2t}{N}$ and $t_{j}=t+j\Delta$, as well as
\[
U(t,T-t)=\lim_{N\rightarrow\infty}e^{-i\Delta H(t_{N-1})}\cdots e^{-i\Delta H(t_{2})}e^{-i\Delta H(t_{0})}.
\]
Since 
\[
t_{N-j}=T-t_{j}
\]
we find
\[
\Gamma e^{-i\Delta H(t_{j})}\Gamma=e^{i\Delta H(-t_{j})}=e^{i\Delta H(T-t_{j})}=e^{i\Delta H(t_{N-j})}
\]
and so 
\begin{align*}
\Gamma U(t,T-t)\Gamma & =\lim_{N\rightarrow\infty}e^{i\Delta H(t_{1})}\cdots e^{i\Delta H(t_{N-1})}e^{i\Delta H(t_{N})}\\
 & =\lim_{N\rightarrow\infty}\left(e^{-i\Delta H(t_{N})}e^{-i\Delta H(t_{N-1})}\cdots e^{-i\Delta H(t_{1})}\right)^{\dagger}\\
 & =U(t,T-t)^{\dagger}.
\end{align*}
Thus $U(t,T-t)$ is a continuous path of unitaries with this symmetry.
Since
\[
\frac{1}{2}\sig\left(U(t,T-t)\Gamma\right)
\]
is constant, we see
\[
\frac{1}{2}\sig(U(T)\Gamma)=\frac{1}{2}\sig(I\Gamma)=0.
\]
\end{proof}
\begin{thm}
If the Floquet Hamiltonian $H(t)$ has chiral symmety $\Gamma H(t)\Gamma=-H(-t)$
for $\Gamma$ with eigenspaces for $-1$ and for $+1$ of the same
dimension, then there is a Hermitian matrix $H$ with 
\[
\Gamma H=-H\Gamma
\]
and
\[
e^{-iH}=U(T).
\]
Moreover, there is an orthonormal basis $\boldsymbol{b}_{1},\dots,\boldsymbol{b}_{N}$
and unit scalars $\lambda_{1},\dots,\lambda_{N}$ with 
$H\boldsymbol{b}_{j}=\lambda_j\boldsymbol{b}_{j}$,
for $j=1,\dots,N$, and
\[
\Gamma\boldsymbol{b}_{j}=\boldsymbol{b}_{\frac{N}{2}+j},
\]
\[
\overline{\lambda_{j}}=\lambda_{\frac{N}{2}+j}
\]
for $j=1,\dots,N/2$. 
\end{thm}

\begin{proof}
Lemma~\ref{lem:Chiral_implies_signature_zero} tells us that Theorem \ref{thm:ChiralLogarithm} applies, so there is a Hermitian
matrix $H$ so that $\Gamma H=-H\Gamma$ and $e^{-iH}=U(T)$. 

The method to find a structured orthonormal basis of eigenvectors for $U(T)$, and
hence for $H$, was described in the proof of Theorem~\ref{thm:Chiral_spectral_decomposition} and the discussion after the proof.

\end{proof}

\section{Symmetry preserving matrix square roots}
\label{section:matrix_sqrt}

\subsection{Matrix preserving functional calculus}

A  matrix $X \in \mathbb{C}^{n\times n}$ with no eigenvalues in $\mathbb{R}^-$ has principal square root $R=X^{1/2}$ defined by applying the principal branch of the scalar square root function via functional calculus.
In particular, for a unitary $U$ with $-1$ not in the spectrum, this is a unitary $V$ with spectrum within the open right half of the unit circle so that $V^2 = U$.  For unitary $V$ with $-1$ in the spectrum, we can look for one of many unitary matrices $V$ with $V^2 = U$ and spectrum in the closed right half of the unit circle.  We start looking at an algorithm that works for unitaries with $-1$ not in the spectrum that tends to preserve symmetries during the calculation.

First, a theorem that explains why it makes sense to search for a matrix square root
algorithm that preserves various symmetries.

\begin{thm}  \label{thm:RootSymmetry}
Assume $X \in \mathbb{C}^{n\times n}$ have no eigenvalues on $\mathbb{R}^{-}$ and
that $R$ is the principal square root of $X$.  Then the following are all true.
\begin{enumerate}
\item If $X$ is unitary then $R$ is unitary.
\item If $X$ is complex symmetric then $R$ is complex symmetric.
\item If $X$ is real then $R$ is real.
\item If $X$ is self-dual then $R$ is self-dual.
\item If $\Gamma X\Gamma=X^{\dagger}$ then $\Gamma R\Gamma=R^{\dagger}$.
\end{enumerate}
\end{thm}

\begin{proof}
Since it is easy, using power series, to see how symmetries in a matrix lead to symmetries
in the exponential of that matrix, these are all easily proven via the formula
\[
R = \exp\left(\frac{1}{2}\log(X)\right)
\]
and Theorem~\ref{thm:FunCalcSymmetry}.
\end{proof}

There are many algorithms to compute a matrix square root \cite{Higham}. One algorithm has been noted to behave well for many
symmetries \cite{HighamEtAlSquareRootGroupPreserving} relies on an iteration
\begin{align}
Y_{k+1}&=\frac{1}{3}Y_k\left[I+8\left(I+3Z_kY_k\right)^{-1}\right],\\
Z_{k+1}&=\frac{1}{3}\left[I+8\left(I+3Z_kY_k\right)^{-1}\right]Z_k.
\label{rec_rel}
\end{align}

Given the initial conditions $Y_{0}=A$ and $Z_{0}=I$ this leads
to good convergence $Y_{k}\rightarrow A^{\frac{1}{2}}$ for $A$ with
spectrum avoiding $\mathbb{R}^{-}$, as is explained in  
\cite[Section 6]{HighamEtAlSquareRootGroupPreserving}.
The following result shows why, in theory at
least, this algorithm will return an approximate square root that
is unitary when the input is unitary, and has an additional symmetry
when the input has that additional symmetry.

\begin{thm}
\label{thm:iterations_symmetries}
Suppose $Y$ and $Z$ are matrices such that $I-3ZY$ is invertible,
and let
\begin{align*}
Y' & =\frac{1}{3}Y\left(I+8\left(I+3ZY\right)^{-1}\right),\\
Z' & =\frac{1}{3}\left(I+8\left(I+3ZY\right)^{-1}\right)Z.
\end{align*}
\begin{enumerate}
\item If $Y$ and $Z$ are unitary then $Y'$ and $Z'$ are unitary.
\item If $Y$ and $Z$ are complex symmetric then $Y'$ and $Z'$ are complex
symmetric.
\item If $Y$ and $Z$ are real then $Y'$ and $Z'$ are real.
\item If $Y$ and $Z$ are self-dual then $Y'$ and $Z'$ are self-dual.
\item If $\Gamma Y\Gamma=Y^{\dagger}$ and $\Gamma Z\Gamma=Z^{\dagger}$
then $\Gamma Y'\Gamma=Y'^{\dagger}$ and $\Gamma Z'\Gamma=Z'^{\dagger}$.
\end{enumerate}
\end{thm}

\begin{proof}
(1) This was proven in \cite{HighamEtAlSquareRootGroupPreserving}.
Note that the scalar identity 
\begin{equation}
\frac{1}{3}\left(1+\frac{8}{1+3z}\right)=\frac{3+z}{1+3z} \label{rat_ident}
\end{equation}
shows that the function sends the unit circle to the unit circle,
so for any unitary matrix $U$ the matrix
\[
\frac{1}{3}\left(I+8\left(I+3U\right)^{-1}\right)
\]
is also unitary.

(2) Assuming $Y^{\top}=Y$ and $Z^{\top}=Z$ we have
\begin{align*}
Y'^{\top} & =\frac{1}{3}\left(I+8\left(\left(I+3ZY\right)^{-1}\right)^{\top}\right)Y\\
 & =\frac{1}{3}\left(I+8\left(I+3YZ\right)^{-1}\right)Y\\
 & =\frac{1}{3}Y\left(I+8\left(I+3ZY\right)^{-1}\right)\\
 & =Y',
\end{align*}
where the last step uses the standard fact about the holomorphic functional
calculus that $f(AB)A=Af(BA)$ so long as $f(AB)$ and $f(BA)$ make
sense. The proof for $Z'$ is similar.

(3) This is obvious.

(4) This is similar to the proof of (2).

(5) Assuming $\Gamma Y\Gamma=Y^{\dagger}$ and $\Gamma Z\Gamma=Z^{\dagger}$
we have 
\begin{align*}
\Gamma Y'\Gamma & =\frac{1}{3}\Gamma Y\Gamma\left(I+8\left(\left(I+3\Gamma Z\Gamma\Gamma Y\Gamma\right)^{-1}\right)\right)\\
 & =\frac{1}{3}Y^{\dagger}\left(I+8\left(I+3\left(Z^{\dagger}Y^{\dagger}\right)\right)^{-1}\right)\\
 & =\left(\frac{1}{3}\left(I+8\left(I+3\left(YZ\right)\right)^{-1}\right)Y\right)^{\dagger}\\
 & =\left(\frac{1}{3}Y\left(I+8\left(I+3\left(ZY\right)\right)^{-1}\right)\right)^{\dagger}\\
 & =Y'^{\dagger}
\end{align*}
 and the proof for $Z'$ is similar. 
\end{proof}

We now consider Algorithm~\ref{alg:theoretical_square_root}, which
is essentially that given in
\cite[Theorem 6.1.]{HighamEtAlSquareRootGroupPreserving}.
Theorem~\ref{thm:iterations_symmetries} shows that 
this iterative algorithm will respect the property of being
unitary while at the same time preserve various symmetries
of importance in physics.

In practice, it is not even possible to start Algorithm~\ref{alg:theoretical_square_root}
as it requires before it starts the condition $UU^\dagger = I$.  This is not likely
to hold in finite-precision arithmetic.  We need to deal with approximate unitaries and
matrices as well as confront approximate symmetries.

\begin{figure}[!tp]
\begin{algorithm}[H]
\begin{flushleft}
\textbf{Input}: unitary matrix $U$ with $-1\protect\notin\sigma(U)$.\\
\textbf{Output}: $V$ a unitary with $V^2\approx U$ with the same symmetry
as $U$. 
\end{flushleft}
\begin{algorithmic}
\STATE $Y\leftarrow U$ 
\STATE  $Z\leftarrow I$
\REPEAT
\STATE  $Y_{0}\leftarrow Y$
\STATE  $Z_{0}\leftarrow Z$
\STATE  $C\leftarrow(1/3)(I+8(I+Z_{0}R_{0})^{-1})$
\STATE  $Y\leftarrow Y_{0}C$
\STATE    $Z\leftarrow CZ_{0}$
\UNTIL{$Y \approx Y_{0}$}
\RETURN  $V=Y$
\caption{The algorithm from \cite[Section 6]{HighamEtAlSquareRootGroupPreserving},
  applied to a unitary that might have an additional symmetry.
  \label{alg:theoretical_square_root}}
\end{algorithmic}
\end{algorithm}
\end{figure}

\subsection{Enforcing symmetries}

We need to confront the fact that $U^\dagger U$ will almost never equal
the identity if we are working in floating point arithmetic.  It is difficult
to do a theoretical analysis on how the error in $Y'$ and $Z'$ being unitary
relates to the error in $Y$ and $Z$ being unitary, where these are defined as
in Theorem~\ref{thm:iterations_symmetries}.  However,  
numerical experimentation with 
\[
U=\left[\begin{array}{cc}
e^{i3.1415926} & 10^{-12}\\
0 & e^{-i3.1415926}
\end{array}\right]
\]
quickly reveals that this is going to be a problem.  Applying the iteration
as in Algorithm~\ref{alg:theoretical_square_root} shows that 
that the error in $Y_{j}$ being unitary grows about one order
of magnitude every two iterations.  One can run the supplemental file
\texttt{test\_generic\_root.m} to see this calculation.  In this situation, the generic
algorithm returns $V$ with 
\[
\|V^{2}-U\|\approx9\times10^{-16},
\]
which is great, except that 
\[
\|V^{\dagger}V-I\|\approx4\times10^{-7},
\]
which is not so great.

Some of the symmetries and relations considered here require effort
to make the algorithms match the theory. For example, when computing structured square roots, Theorem \ref{thm:iterations_symmetries} tells us that if $Y_{0}$ and $Z_{0}$ are symmetric unitaries then
\[
Y_{1}=\frac{1}{3}Y_{0}\left[I+8\left(I+3Z_{0}Y_{0}\right)^{-1}\right]
\]
 is again a symmetric unitary. In practice, even if the relations
\[
Y_{0}^{\top}=Y_{0},\ Z_{0}^{\top}=Z_{0},\ Y_{0}Y_{0}^{\dagger}=I\ Z_{0}{Z_{0}}^{\dagger}=I
\]
 hold exactly, numerical errors mean that we only have 
\[
Y_{1}^{\top}\approx Y_{1},\ Y_{1}{Y_{1}}^{\dagger}\approx I.
\]

Left unchecked, these errors will grow as the iteration continues. 
A fix that works well here is to adjust the intermediate matrices
to be closer to unitary and exactly satisfy other symmetries if
appropriate.

We can easily enforce symmetry. If $U^{\top}$ is close to $U$ we
just replace a matrix $U$ by 
\begin{equation} \label{eqn:enforce_complex_symmetric}
U'=\frac{1}{2}\left(U+U^{\top}\right).
\end{equation}

A similar averaging will enforce the symmetries such as $U^{\dagger}=\Gamma U\Gamma$. If we regard ``being
real'' as a symmetry, this is generally enforced by using real floating
point arithmetic instead of complex floating point arithmetic.

\begin{figure}[!tp]
\begin{algorithm}[H]
\begin{flushleft}
\textbf{Input}: matrix $U$ with $-1$ not close to $\sigma(U)$ and $U^{\dagger}U\approx I$ \\
\textbf{Output}: $V$ a matrix with $V^{\dagger}V\approx I$ and  $V^2 \approx U$ with the same symmetry
as $U$. 
\end{flushleft}
\begin{algorithmic}
\STATE $Y\leftarrow U$ 
\STATE  $Z\leftarrow I$
\REPEAT
\STATE  $Y_{0}\leftarrow Y$
\STATE  $Z_{0}\leftarrow Z$
\STATE  $C\leftarrow(1/3)(I+8(I+Z_{0}R_{0})^{-1})$
\STATE  $Y\leftarrow Y_{0}C$
\STATE    $Z\leftarrow CZ_{0}$
\STATE $Y \leftarrow(1/2)(Y+(Y^\dagger)^{-1})$
\STATE $Z \leftarrow(1/2)(Z+(Z^\dagger)^{-1})$
\STATE Enforce any needed symmetry on $Y$
\STATE Enforce any needed symmetry on $Z$
\UNTIL{$Y \approx Y_{0}$}
\RETURN  $V=Y$
\caption{Symmetry enforcing algorithm for computing a matrix square root.
  \label{alg:correcting_square_root}}
\end{algorithmic}
\end{algorithm}
\end{figure}

The relation $U^{\dagger}U=I$ is a bit different. To begin with, it is
highly unlikely that this literally holds. The best we can hope for
is something like 
$\left\Vert U^{\dagger}U-I\right\Vert \leq10^{-12}$
in a suitable norm. This error can grow as the iteration continues.
In practice, we can reduce an already small error to be even smaller
by replacing in intermediate matrix $V$ by
\[
\frac{1}{2}\left(V+\left(V^{\dagger}\right)^{-1}\right)
\]
which is one step in a basic implementation of a Newton type method
for computing the polar part of an invertible matrix \cite[\S3.2]{HighamPolarDecomposition}. In practice, we find this suffices to compensate for
the non-unitary nature of the numerical errors in the iterative algorithm
for computing a matrix square root of a unitary matrix.

The resulting algorithm is listed as Algorithm~\ref{alg:correcting_square_root}.
We have only tested this carefully for the case of $U$ being a generic unitary, complex symmetric unitary or a unitary matrix
with the Chiral symmetry.  It should adapt to all manner of symmetry,
so long as one has an averaging formula like
Equation~\ref{eqn:enforce_complex_symmetric} to enforce that symmetry.

We provide all the Matlab programs needed to recreate the
data in the figures as supplementary files  \cite{LoringVidesSuppl}.
A basic implementation of Algorithm~\ref{alg:correcting_square_root} 
is in supplementary file \texttt{sqrtu\_basic.m}.  This implementation
does not check the input in any way.  In particular, if called on a class
AIII unitary with nontrivial index it will fail.

\begin{figure}[!tp]
\includegraphics{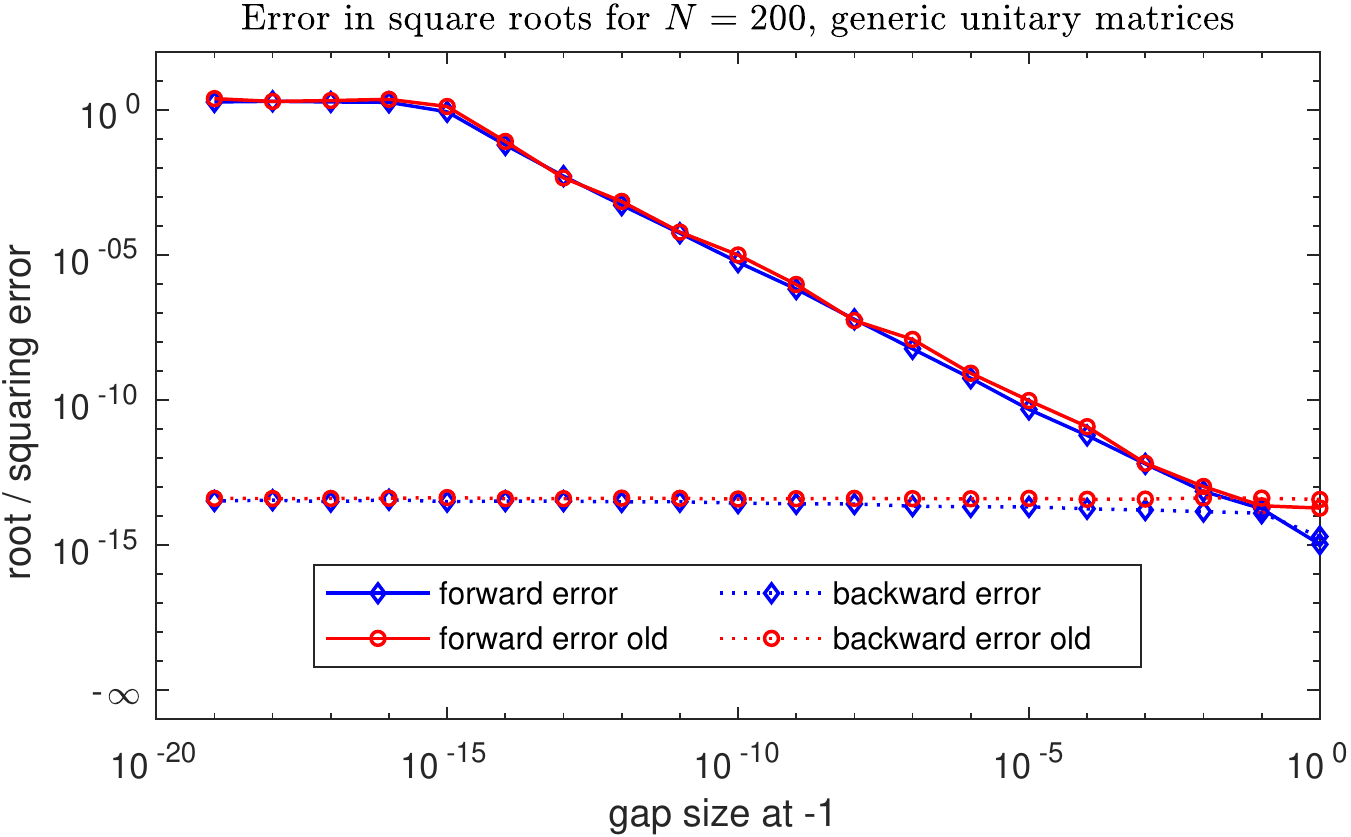}
\caption{Accuracy of the new algorithm is the solid blue line with diamonds.
Accuracy of standard (Matlab) algorithm is the solid red with circles.
Error in square of the root versus square of the original, the dotted
blue line for the new algorithm, the dotted red line with diamonds
for the standard algorithm. This is for generic complex unitary matrices,
with a gap at $-1$ in the spectrum of variable sizes. The test matrices
here are with $N=200$. 
\label{fig:Squaring-a-unitaryI}}
\end{figure}

\begin{figure}[!tp]
\includegraphics[clip]{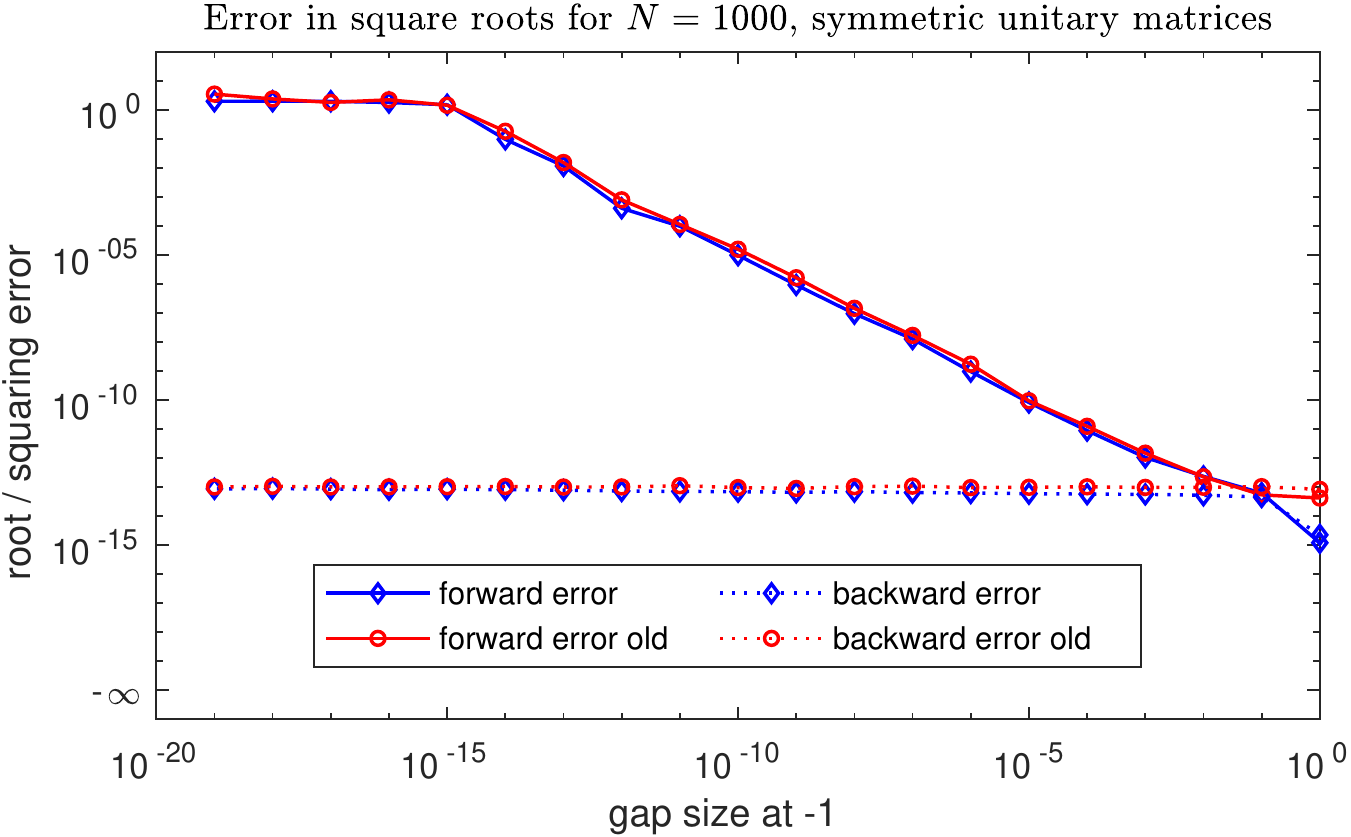}
\caption{As in Figure~\ref{fig:Squaring-a-unitaryI} but now for symmetric
unitary matrices, with a gap at $-1$ in the spectrum of variable
sizes and with $N=1000$. 
\label{fig:fig:Squaring-a-unitaryII}}
\end{figure}

\subsection{Forward and Backward stability}
\label{subsec:AccuracyOfRoot}

We compare forward and backward stability of the new algorithm to
the stability of a standard matrix square root algorithm, \texttt{sqrtm}
in Matlab. Since the standard branch of the logarithm is discontinuous
at $-1$ we expect forward stability to lessen as the gap at $-1$
closes.

Figure~\ref{fig:Squaring-a-unitaryI} and Figure~\ref{fig:fig:Squaring-a-unitaryII}
show the average, computed across ten test unitaries, of forward and
backward error.  One works with generic  unitary matrices, the other
with symmetric unitary matrices.  The matrix size and symmetry 
class have little effect here.  
To make the task a bit more complicated, the matrices
we selected have four eigenvalues at the prescribed arc distance
to $-1$.  The square root of the matrix was formed first, then
that was squared to be the input to the algorithms.  

The takeaway from these tests are that our algorithm is about as good in forward and backward stability as a well-established algorithm.  The
improvement is in the output being unitary, and having other desired symmetries, when the input  is unitary.

\subsection{Preservation of symmetries}

\begin{figure}[!tp]
\includegraphics[clip]{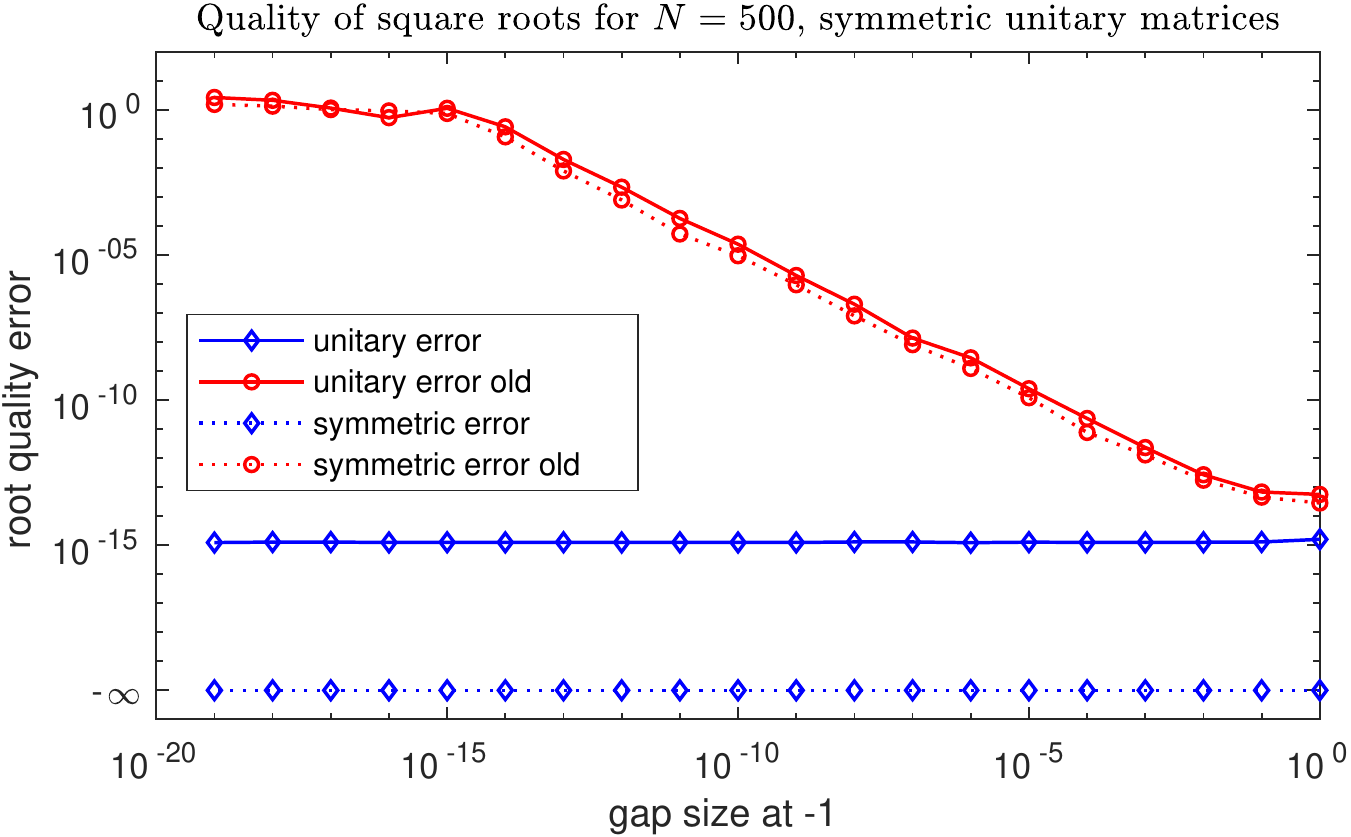}
\caption{For $N=500$ and symmetric unitary matrices. Blue with diamonds indicates
the new algorithm, red with circles indicates a standard algorithm
from Matlab. Solid line indicates error in being unitary. The dotted
lines indicate the error in being symmetric (0 is plotted as $5\times10^{-20}$
to make it visible in a log-log plot. 
\label{fig:fig:Quality_sqr_rt_I}}
\end{figure}

\begin{figure}[!tp]
\includegraphics[clip]{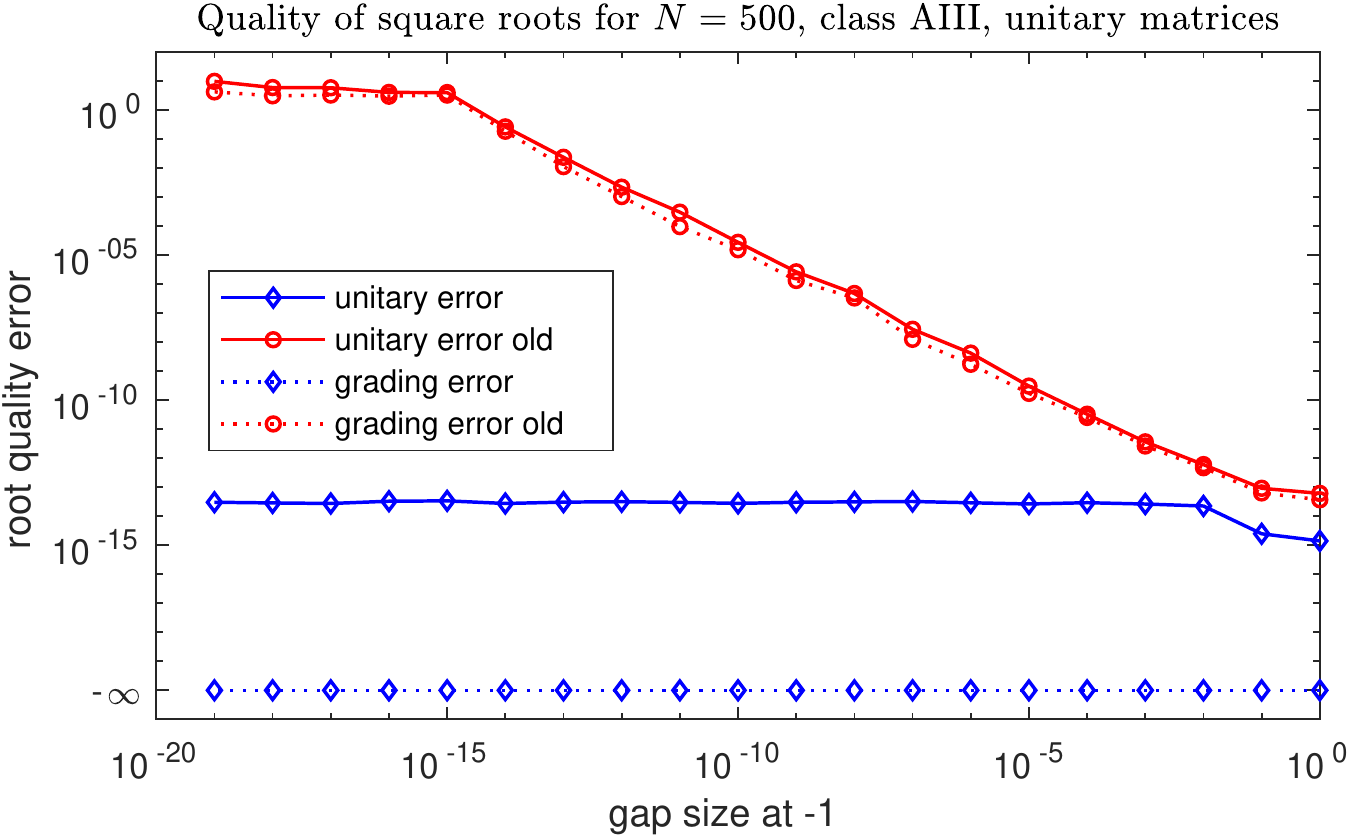}
\caption{Same as in Figure~\ref{fig:fig:Quality_sqr_rt_I}, but for $N=500$
and AIII symmetric unitary matrices. The dotted lines indicate the
error in the relation $\Gamma U\Gamma=U^{\dagger}$. 
\label{fig:fig:Quality_sqr_rt_II}}
\end{figure}

\begin{figure}[!tp]
\includegraphics[clip]{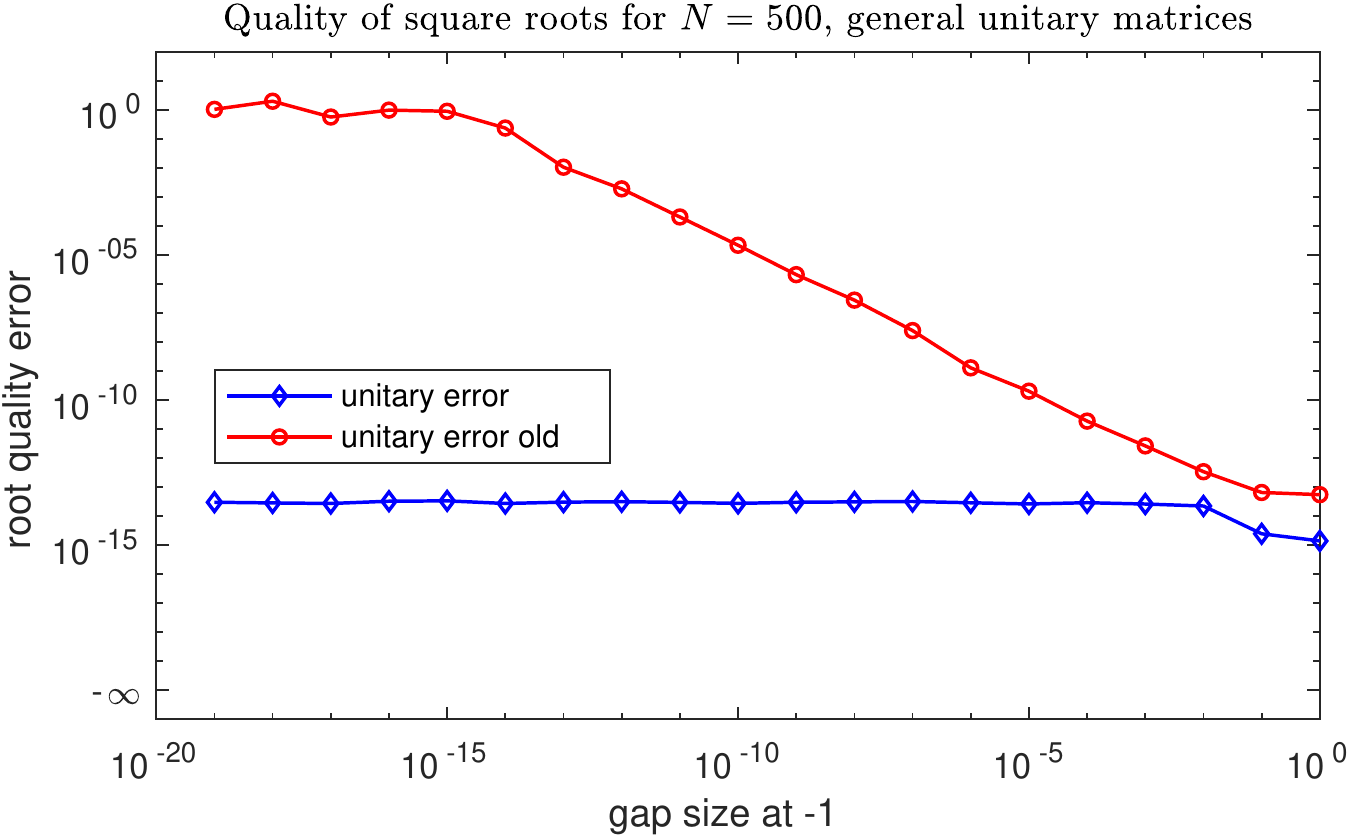}
\caption{Same as in Figure~\ref{fig:fig:Quality_sqr_rt_I}, but for $N=500$
and general complex unitary matrices. There is no second relation
to check. 
\label{fig:fig:Quality_sqr_rt_III}}
\end{figure}

We want our square roots to be unitary, and perhaps approximately
satisfy an additional symmetry.  In Figures~\ref{fig:fig:Quality_sqr_rt_I}-\ref{fig:fig:Quality_sqr_rt_III} we see a huge improvement in
our algorithm relative the generic algorithm. Some of the errors
in the new algorithm are zero, which get plotted as $5\times10^{-20}$
to make them show up on the log-log plot.

One might consider just taking the output $V$ of the standard algorithm and enforce on that
the extra symmetries and then computing the polar decomposition. In
cases were $\left\Vert VV^{\dagger}-I\right\Vert \approx10$ (happens
frequently) it may be that $V$ is singular and there is not a well
defined unitary polar part. This approach seems impossible.

\subsection{Nudging away from the fixed point with disorder}

We will present  in \S\ref{subsec:AccuracyOfRoot} data on the accuracy of Algorithm~\ref{alg:correcting_square_root}.  The upshot is that
it  is  very accurate, even if the spectrum of $U$ includes a point very 
close to $-1$.
  
In theory, if the input unitary $U$ has an eigenvalue at, or extremely close
to $-1$, then Algorithm~\ref{alg:correcting_square_root} should fail to
converge.  This is a rare problem, in practice, unless there is an actual
obstruction as in the chiral case.  It seems that round-off errors are
enough to have the effect of nudging eigenvalues away from $-1$.
Of course, if the index in Equation~\ref{eqn:AIII_index} does not vanish,
then there will be always at least one eigenvalue stuck at $-1$.  
This ``topological  protection'' will mean no structured square root exists,
so the algorithm was expected fail in this situation.

It should be easy to program so as to avoid even the rare case where
numerical errors are not enough to make the eigenvalue at $-1$ behave
as if it is not really at $-1$.  After a number of iterations in 
Algorithm~\ref{alg:correcting_square_root} one could add a random
matrix to $R_0$ that is small and mostly living in the eigenspaces
for $R_0$ near $-1$. Alternately, one could just re-run
Algorithm~\ref{alg:correcting_square_root} with a unitary that is created as the unitary close to small random perturbation of the original unitary.

\section{Symmetry preserving matrix logarithms}

\label{section:matrix_log}

\subsection{Inverse scaling and squaring}

In this section we will build on the ideas presented in \cite[\S2.3.]{arslan2017functions}. We will adapt the structured inverse scaling and squaring method proposed in \cite{arslan2017functions} to derive an approximation method of $\log(U)$ for any symmetric unitary matrix $U\in \mathbb{C}^{n\times n}$. 

Given a unitary matrix $U$, perhaps with an additional symmetry, the structured approximation method proposed in this section will be obtained by 	taking repeated 
square roots via Algorithm~\ref{alg:correcting_square_root}, followed by
applying the $[n/n]$ Pad\'e approximant $r_n$ of the function $\log(1+x)$ for a
modest sized $n$.  

In order to estimate the approximation error of the matrix logarithm computation one can apply \cite[Thm. 11.6]{Higham}.  This states that for any subordinate matrix norm,
\[
\|r_n(X)-\log(I+X)\|\leq |r_n(-\|X\|)-\log(1-\|X\|)|. 
\]
 given any matrix $X$ with $\|X\|<1$.

After taking square root five times of a unitary $U$, we can expect a
matrix $V$
with spectrum on the arc between $e^{\pm \pi i/32}$.  Assuming some
numerical
errors, we still may assume the spectrum of $V$ will be within $0.1$ of $1$.
A numerical plot of
\[
x\mapsto r_7(-x) - \log(1-x)
\]
over the interval $[-0.1, 0.1]$  shows that the error, in the spectral
norm, of
applying $r_7(\lambda)$ to $V-I$ instead of $\log(1+\lambda)$ will be
less than
$1.3\times 10^{-16}$.

\begin{figure}[!tp]
\begin{algorithm}[H]
\begin{flushleft}
\textbf{Input}: matrix $U$ with $-1$ not close to $\sigma(U)$ and $U^{\dagger}U\approx I$ \\
\textbf{Output}: $H$ a matrix with $H^{\dagger} = -H $ and  $\exp(H) \approx U $ with the same symmetry
as $U$. 
\end{flushleft}
\begin{algorithmic}
\STATE $R\leftarrow U^{1/32}$  using Algorithm \ref{alg:correcting_square_root} five times.
\STATE  $H \leftarrow r_7(R-I)$
\STATE $ H \leftarrow \frac{1}{2}(H - H^\dagger)$
\STATE $ H \leftarrow 32H$
\STATE Enforce any needed symmetry on $H$
\caption{Symmetry enforcing algorithm for computing a matrix logarithm.
  \label{alg:log_u_basic}}
\end{algorithmic}
\end{algorithm}
\end{figure}

Our algorithm is then to apply our symmetry preserving square root algorithm
five times, use $r_7$ to get a log of that, and rescale.  One can vary this
and take $k$ root and use $r_n$ but one cannot increase both $n$ and $k$
by much without introducing numerical errors that overwhelm the increase
in accuracy in the Pad\'e approximant.  We shall see that the resulting logarithm
performs about as well as can be expected when using double floating point
arithmetic.

Our algorithm for structured logarithms is summarized as
Algorithm~\ref{alg:log_u_basic}.  This is implemented in the
supplementary file \texttt{logsu\_basic.m} in \cite{LoringVidesSuppl}.

\subsection{Forward and backward stability}

\begin{figure}[!tp]
\includegraphics{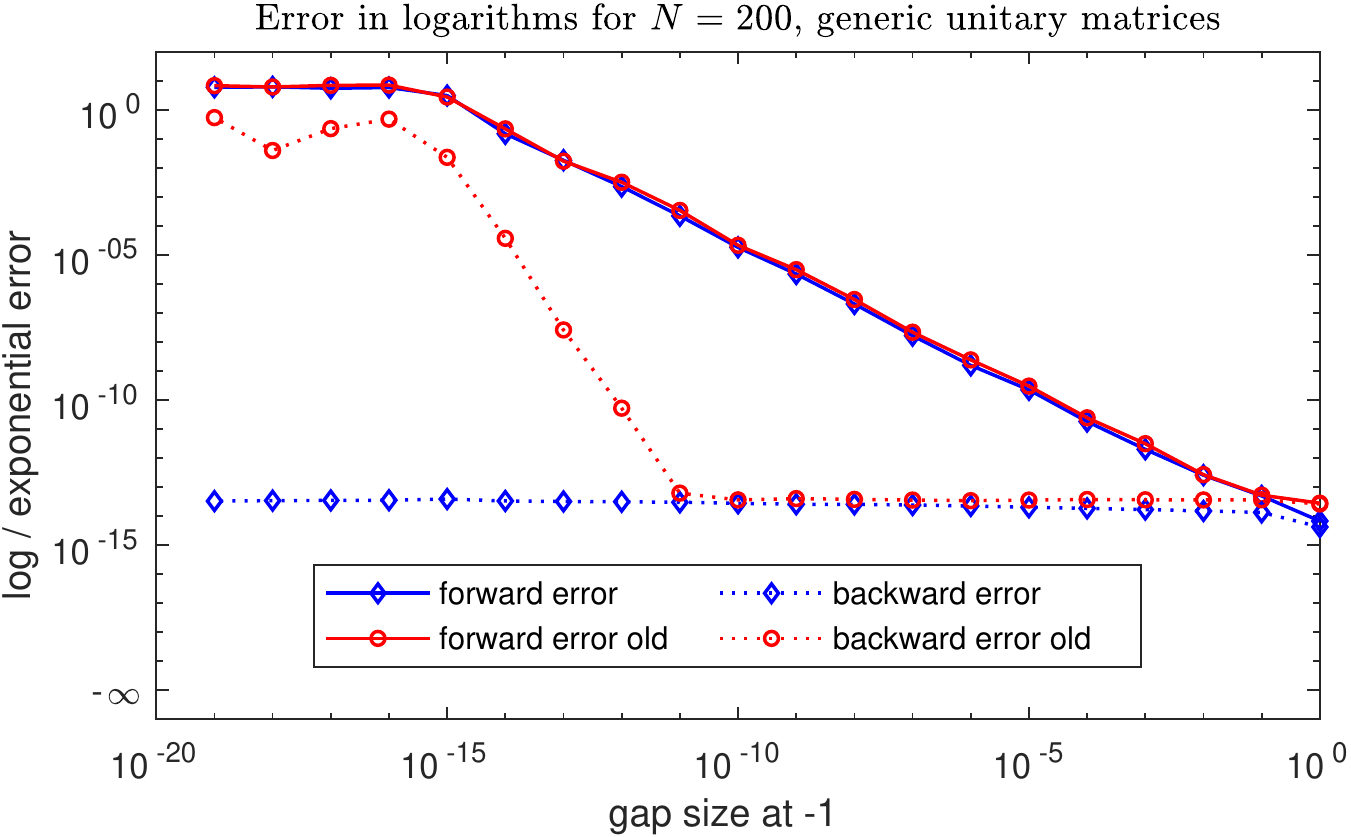}
\caption{Comparing foward and backward stability of 
Algorithm~\ref{alg:log_u_basic} with the built in Matlab
function \texttt{logm}.  This is for test unitaries constructed
as for the root test, with no additional symmetry for
unitary matrices with $N=200$.
\label{fig:Accuracy_log_I}}
\end{figure}

\begin{figure}
\includegraphics[clip]{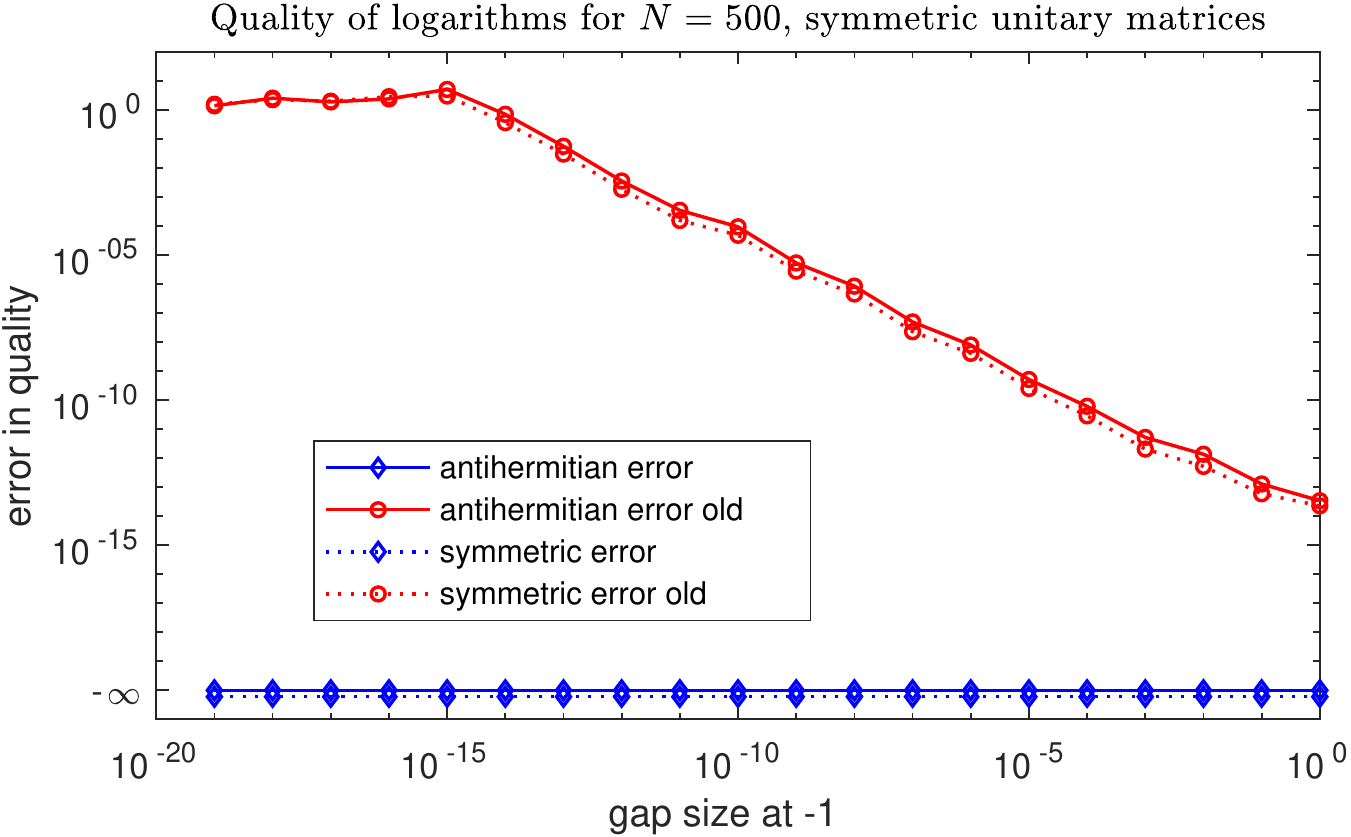}
\caption{For $N=500$ and symmetric unitary matrices. Blue with diamonds indicates
the new algorithm, red with circles indicates a standard algorithm
from Matlab. Solid line indicates error in being antihermitian. The
dotted lines indicate the error in being symmetric (0 is plotted as
$5\times10^{-20}$ to make it visible in a log-log plot. 
\label{fig:Quality_log_I}}
\end{figure}

\begin{figure}[!tp]
\includegraphics[clip]{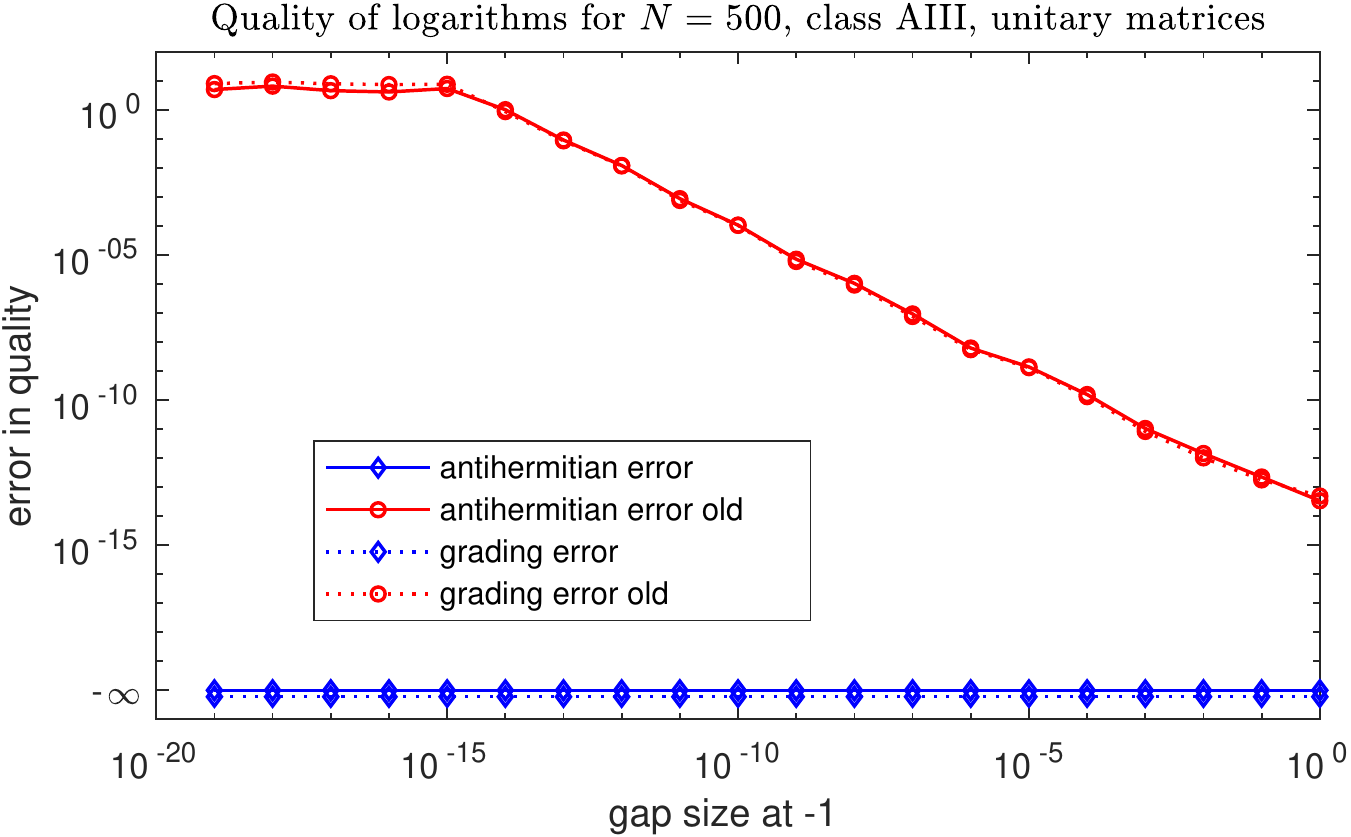}
\caption{Same as in Figure~\ref{fig:Quality_log_I}, but for $N=500$ and
AIII symmetric unitary matrices. The dotted lines indicate the error
in the relation $\Gamma H\Gamma=-H$ for $H$ the approximation to
$\log(U)$. 
\label{fig:Quality_log_II}}
\end{figure}

\begin{figure}[!tp]
\includegraphics[clip]{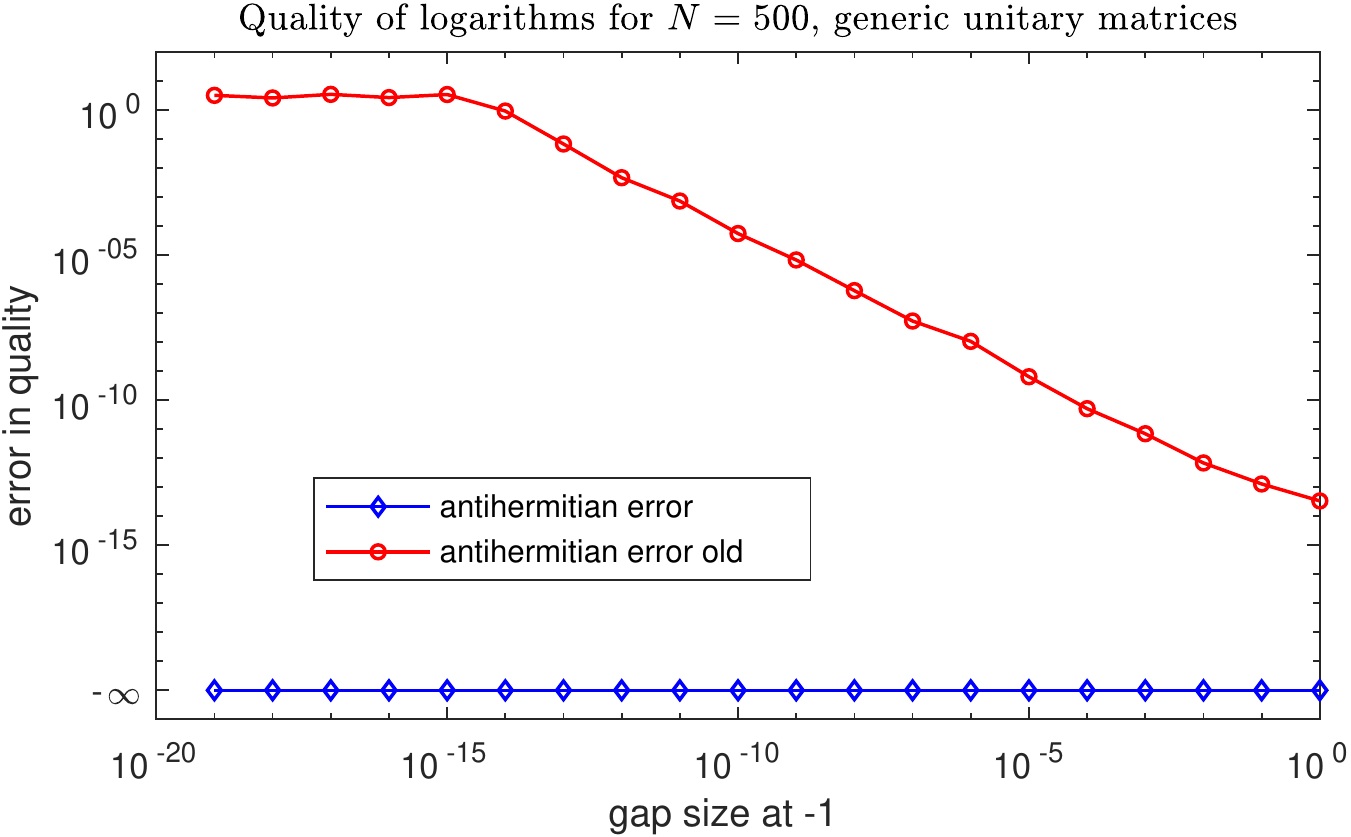}
\caption{Same as in Figure~\ref{fig:Quality_log_I}, but for $N=500$ and
general complex unitary matrices. There is no second relation to check.
\label{fig:Quality_log_III}}
\end{figure}

After the first square root is calculated, the rest of the algorithm involves
computing functional calculus applied to matrices whose spectrum stays
well away from any point of discontinuity.   The numerical study we
preformed  on Algorithm~\ref{alg:log_u_basic} gives data that looks a lot
like the data on the accuracy of Algorithm~\ref{alg:correcting_square_root}
so we omit most of it.  See Figure~\ref{fig:Accuracy_log_I} for a some
of the data.  Again, if we ignore the need for the answer being unitary, then
the new algorithm is roughly comparable to a standard algorithm, although
it does improve on the standard algorithm on backward error when the
gap at $-1$ is small.

\subsection{Preservation of symmetry}

We want our logarithms to be anti-Hermitian, and perhaps approximately
satisfy an additional symmetry. Here we see a huge improvement in
our algorithm relative the the generic algorithm. Some of the errors
in the new algorithm are zero, which get plotted as $5\times10^{-20}$
to make them show up on the log-log plot.  See Figures~\ref{fig:Quality_log_I}-\ref{fig:Quality_log_III}.

One might just take the output $V$ of the standard algorithm, enforce
the extra symmetries and then compute the polar decomposition. In
cases were $\left\Vert VV^{\dagger}-I\right\Vert \approx10$ (happens
frequently) it may be that $V$ is singular and there is not a well
defined unitary polar part. This approach seems impossible.

\section{Symmetry preserving matrix diagonalization}
\label{sec:MatrixDiagAlg}

\subsection{Diagonalize a unitary via its log}

In this section we will consider a symmetry preserving matrix diagonalization method, for unitary matrices with time reversal and Chiral symmetry.  For any
unitary $Q$ we have
\[
 Qe^U Q^\dagger =  e^{QU Q^\dagger},
\]
so the theory is telling us that if we do a good job finding a logarithm
$L$ of a unitary matrix $U$ then we can diagonalize $U$ by diagonalizing
$L$.  We can also diagonalize $H=-iL$.

The advantage of working with $H$ where $e^{iH}=U$ is that $H$ will be
Hermitian, and there are many available algorithms for finding an
\emph{orthogonal} basis of eigenvectors of Hermitian matrices.
Unfortunately, there are not easily available algorithms available for a
 structured diagonalization of a Hermitian matrix with the a added symmetries
corresponding to some of the Atland-Zirnbauer symmetry classes
\cite{altland1997nonstandard}.
A very adaptable method of diagonalization is the Jocobi method.
Since this applies to all normal matrices \cite{Wilkinson_Normal_Jacobi},
one might just as well compute directly on $U$ so we do not discuss this
method further.

In the three symmetry classes we have focused on, there are
rather standard methods to use on $H$.  In the general complex
case $H$ is just complex and Hermitian.  In the complex symmetric
case for $U$, we find $H$ to be real and symmetric.  In both cases,
there are very standard algorithms to use.  

When $U$ is in class
AIII, then we find $H$ to be odd, in that $\Gamma H = -H \Gamma$, as
well as Hermitian.  Thus
\[
H=\left[\begin{array}{cc}
0 & A\\
A^{\dagger} & 0
\end{array}\right]
\]
for a matrix $A$.
There is a well known connection between the
eigenvalues of $H$ and the sigular values of $A$. 
See \cite[\S 8.6.1]{GolubVanLoanMatrixComp}, for example.
For the current setting, we know $A$ is square.  If
$A=UDV^{\dagger}$ for $D$ diagonal and positive and $U$ and $V$
unitary, then
\[
\frac{1}{\sqrt{2}}\left[\begin{array}{cc}
U & U\\
-V & V
\end{array}\right]\left[\begin{array}{cc}
-D & 0\\
0 & D
\end{array}\right]
\left(\frac{1}{\sqrt{2}}
\left[\begin{array}{cc}
U & U\\
-V & V
\end{array}\right]
\right)^{\dagger}=\left[\begin{array}{cc}
0 & A\\
A^{\dagger} & 0
\end{array}\right] .
\]

Thus for $H$ given that is Hermitian and odd,
any algorithm that computes the full SVD of the top-right corner of
$H$ will give also the a structured diagonalization of $H$.  Specifically,
$\Gamma$ times the $j$th column of the unitary
\[
\frac{1}{\sqrt{2}}\left[\begin{array}{cc}
U & U\\
-V & V
\end{array}\right]
\]
equals the column in position $j+N/2$ so the orthogonal basis of eigenvectors
occurs in the desired sort of pairs.

\begin{figure}[!tp]
\begin{algorithm}[H]
\begin{flushleft}
\textbf{Input}: matrix $U$ with $-1$ not close to $\sigma(U)$ and $U^{\dagger}U\approx I$ \\
\textbf{Output}: $Q$ and $D$ with $Q$  structured unitary and $D$ unitary and diagonal, such that $U=QDQ^\dagger$. 
\end{flushleft}
\begin{algorithmic}
\STATE Compute $L$ with $e^H=U$, with $L$ anti-Hermitian and with
needed extra symmetries. 
\STATE $H \leftarrow -i H$
\STATE Compute $Q$ and $E$ with $H=QEQ^\dagger$, with $E$ real symmetric and diagonal, and with $Q$ a unitary with appropriate structure. 
\STATE $D\leftarrow e^{iE}$
\RETURN  $Q,D$
\caption{Symmetry enforcing algorithm for diagonalizing a unitary matrix.
  \label{alg:diagonalize_U}}
\end{algorithmic}
\end{algorithm}
\end{figure}

Our algorithm based on these ideas is summarized as
Algorithm~\ref{alg:diagonalize_U}.
 This is implemented in the
supplementary file \texttt{struct\_unitary\_diag.m} in \cite{LoringVidesSuppl}.

\subsection{Testing the diagonalization of unitaries}

We look at the quality of the diagonalization found by
Algorithm~\ref{alg:log_u_basic} on random unitaries in
three symmetry classes.  We compare this to
applying the standard Matlab eigensolver to each unitary
matrix.  As that algoroithm is designed to work in nonnormal
matrices, it is not surprising it does a terrible job of
finding orthogonal eigenvectors.
Our results are summarized in Figures~\ref{fig:Quality_diag-I}-\ref{fig:Quality_diag-III}.

\begin{figure}
\includegraphics[clip]{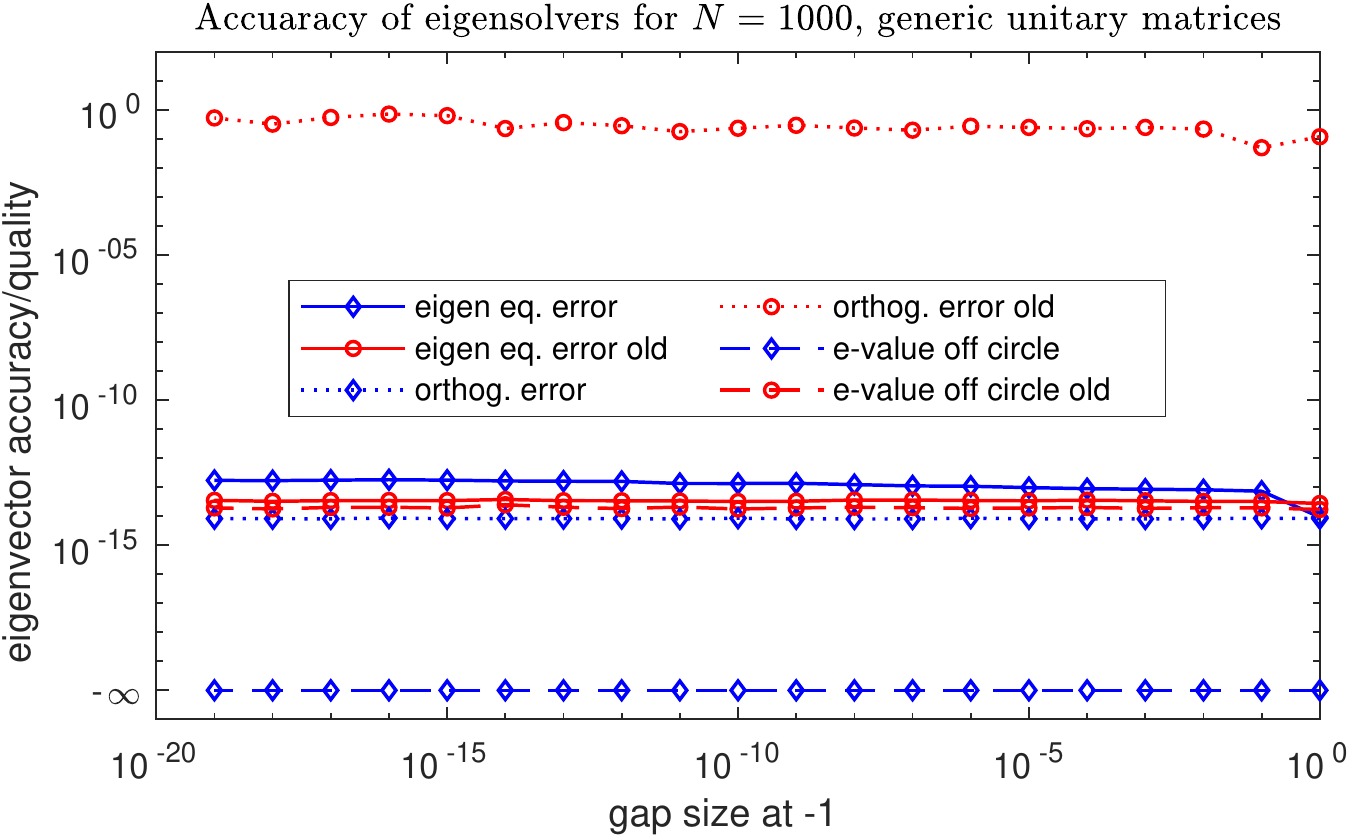}
\caption{Looking at the new diagonalization compared to a generic matrix diagonalization.
Diamond markers are the new algorithm, circle from the the generic
Matlab eigensolver. Solid is accuracy of $U\boldsymbol{v}=\lambda\boldsymbol{v}$.
Dotted is distance from the basis being orthogonal. Dashed is for
distance eigenvalues computed are away from the uni circle. This is
for generic complex unitary matrices, with $N=1000$. 
\label{fig:Quality_diag-I}}
\end{figure}

\begin{figure}
\includegraphics[clip]{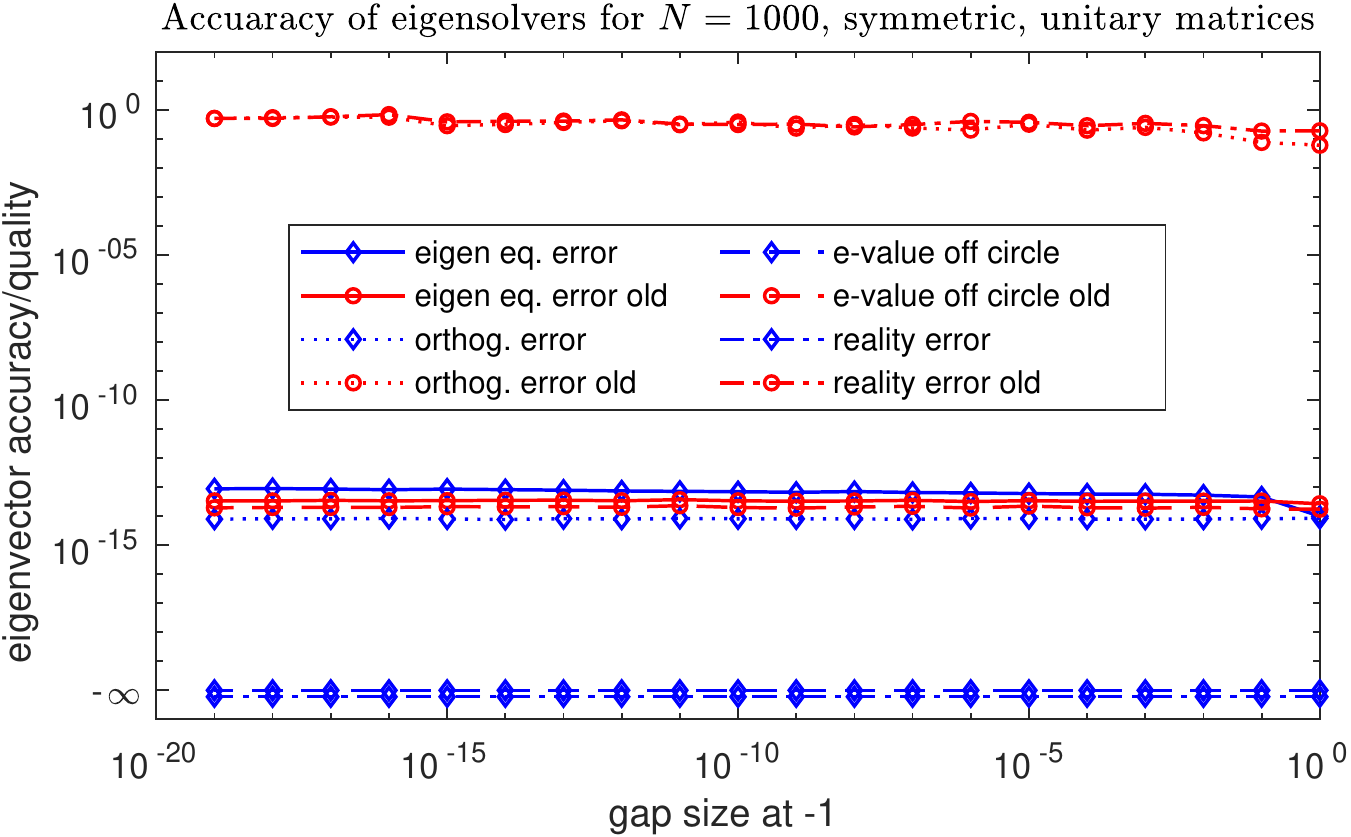}
\caption{As in Figure~\ref{fig:Quality_diag-I} except, this is for symmetric
complex unitary matrices, with $N=1000$, and there are the additional
plots, dash-dot, for the distance of the basis from being real. 
\label{fig:Quality_diag-II}}
\end{figure}

\begin{figure}
\includegraphics[clip]{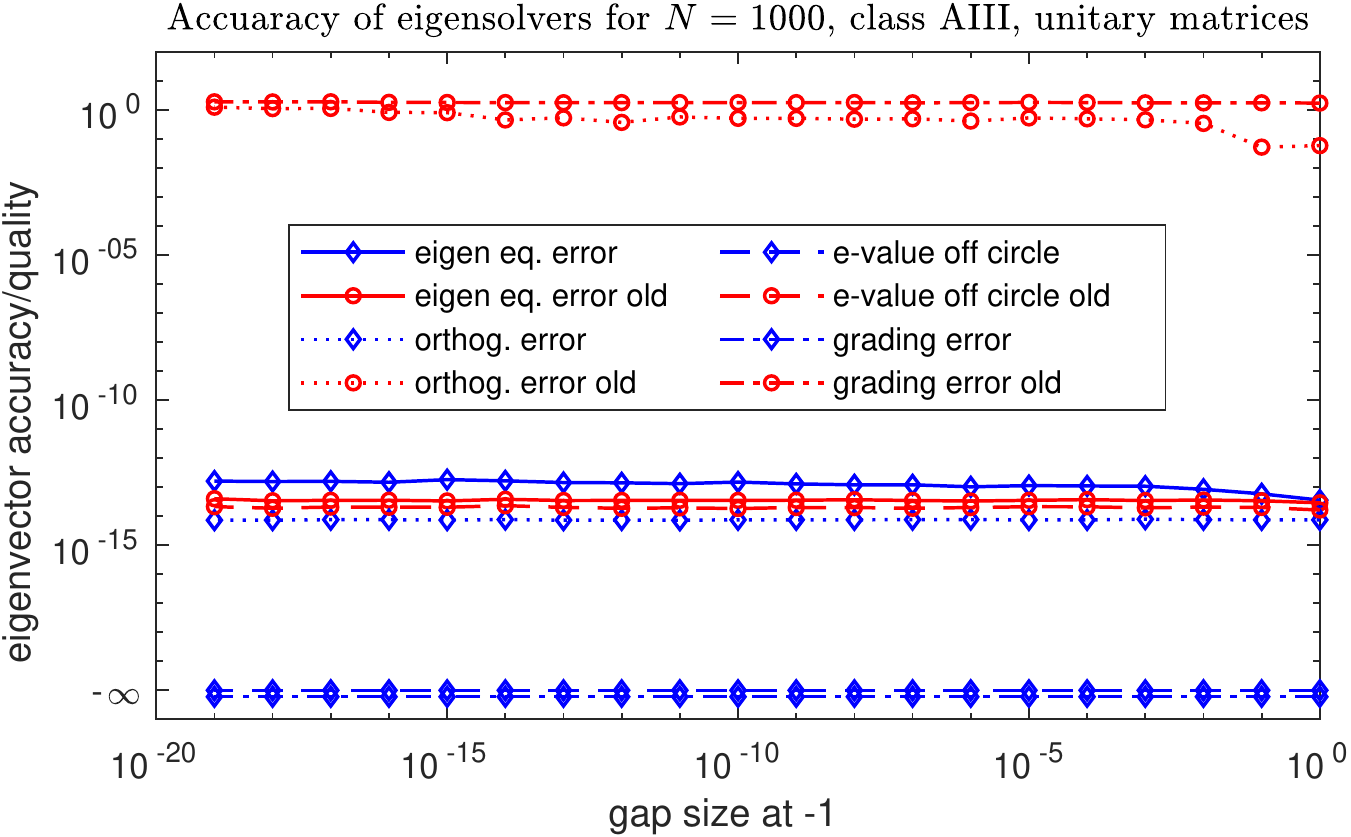}
\caption{As in Figure~\ref{fig:Quality_diag-I} except, this is for class
AIII complex unitary matrices, with $N=1000$, and there are the additional
plots, dash-dot, for the distance of the basis from being in time-reversal
pairs. 
\label{fig:Quality_diag-III}}
\end{figure}

Finally, we document the time needed for the structured unitary eigensolver.
As expected, it behaves as roughly order $N^3$ in time and seems to be
slower than the standard algorithm by a constant factor.  We show results
only in the complex case, in Figure~\ref{fig:Time_diag-I},  but the other cases
look similar. 

\section*{Acknowledgements}

This material is based upon work supported by the National Science Foundation under Grant No. DMS 1700102.

\begin{figure}
\includegraphics[clip]{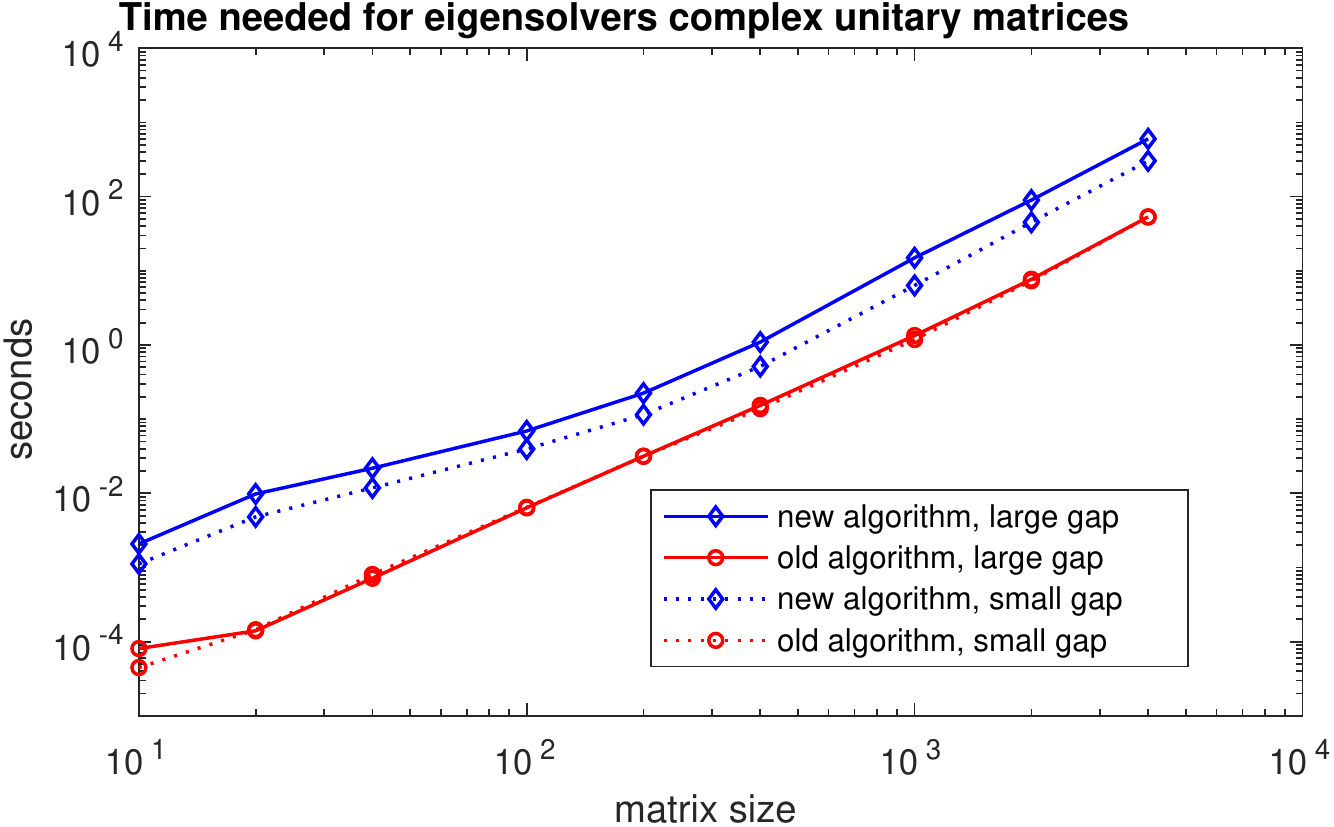}
\caption{Time in seconds for the new eigensolver for unitary matrices, contrasted
with the time needed for a generic, unstructured eigensolve. Here
for complex unitary matrices of various sizes and gaps at $-1$.
The large gaps are of size $10^{-2}$ and the small gaps of
size $10^{-15}$.
\label{fig:Time_diag-I}}
\end{figure}

\clearpage

\bibliographystyle{plain}
\bibliography{AIPFloquetHamiltoniansWithSymmetry}

\end{document}